\documentclass[final]{siamltex}

\usepackage{graphicx}
\usepackage{color}
\usepackage{bm}
\usepackage{amsfonts}
\usepackage{amsmath}

\newcommand{\tr}{^{\sf T}}
\newcommand{\m}[1]{{\bf{#1}}}
\newcommand{\g}[1]{\mbox{\boldmath $#1$}}
\newcommand{\C}[1]{{\cal {#1}}}
\renewcommand{\bar}{\overline}
\input epsf

\newenvironment{nscenter}
 {\parskip=0pt\par\nopagebreak\centering}
 {\par\noindent\ignorespacesafterend}
\newtheorem{remark}{Remark}[section]

\title{
The Switch Point Algorithm
\thanks{January 26, 2021, revised April 28, 2021.
This research has been partially supported by the National
Science Foundation research grants 1819002 and 2031213.}
}

\author{
    Mahya Aghaee\thanks{{\tt mahyaaghaee@ufl.edu},
        https://people.clas.ufl.edu/mahyaaghaee/,
        PO Box 118105,
        Department of Mathematics,
        University of Florida, Gainesville, FL 32611-8105.}
\and
    William W. Hager\thanks{{\tt hager@ufl.edu},
        http://people.clas.ufl.edu/hager/,
        PO Box 118105,
        Department of Mathematics,
        University of Florida, Gainesville, FL 32611-8105.}
}

\begin{document}

\maketitle

\begin{abstract}
The Switch Point Algorithm is a new approach for solving optimal
control problems whose solutions are either singular or bang-bang
or both singular and bang-bang,
and which possess a finite number of jump discontinuities in an
optimal control at the points in time where the solution structure changes.
Problems in this class can often be reduced to an
optimization over the switching points.
Formulas are derived for the derivative of the objective with respect to the
switch points, the initial costate, and the terminal time.
All these derivatives can be computed simultaneously in just one integration
of the state and costate dynamics.
Hence, gradient-based unconstrained optimization techniques, including
the conjugate gradient method or quasi-Newton methods,
can be used to compute an optimal control.
The performance of the algorithm is illustrated using test problems
with known solutions and comparisons with other algorithms from the literature.
\end{abstract}



\pagestyle{myheadings}
\thispagestyle{plain}
\markboth{M. AGHAEE and W. W. HAGER}{SWITCH POINT ALGORITHM}

\section{Introduction}
%
%
Let us consider a fixed terminal time control problem of the form
\begin{equation}\label{CP}
\min C(\m{x}(T)) \quad \mbox{subject to} \quad \m{\dot{x}}(t) =
\m{f}(\m{x}(t), \m{u}(t)), \quad \m{x}(0) = \m{x}_0, \quad
\m{u}(t) \in \C{U}(t),
\end{equation}
where $\m{x} : [0, T] \rightarrow \mathbb{R}^n$ is absolutely continuous,
$\m{u} : [0, T] \rightarrow \mathbb{R}^m$ is essentially bounded,
$C : \mathbb{R}^n \rightarrow \mathbb{R}$,
$\m{f} : \mathbb{R}^n \times \mathbb{R}^m \rightarrow$ $\mathbb{R}^n$, and
$\C{U}(t)$ is a closed and bounded set for each $t \in [0, T]$.
The dynamics $\m{f}$ and the objective $C$ are assumed to be differentiable.
The costate equation associated with (\ref{CP}) is the linear
differential equation
\begin{equation}\label{p}
\m{\dot{p}}(t) =
-\m{p}(t)\nabla_x \m{f}(\m{x}(t), \m{u}(t)),
\quad \m{p}(T) = \nabla C(\m{x}(T)) ,
\end{equation}
where $\m{p} : [0, T] \rightarrow \mathbb{R}^n$ is a row vector,
the objective gradient $\nabla C$ is a row vector, and
$\nabla_x \m{f}$ denotes the Jacobian of the dynamics with respect to $\m{x}$.
At the end of the Introduction, the notation and terminology are summarized.
Under the assumptions of the Pontryagin minimum principle,
a local minimizer of (\ref{CP}) and the associated costate have the
property that
\[
H (\m{x}(t), \m{u}(t), \m{p}(t)) =
\inf \{ H (\m{x}(t), \m{v}, \m{p}(t)) : \m{v} \in \C{U}(t) \}
\]
for almost every $t \in [0,T]$, where
$H(\m{x}, \m{u}, \m{p}) = \m{p}\m{f}(\m{x}, \m{u})$ is the Hamiltonian.

The Switch Point Algorithm is well suited for problems where an
optimal control is piecewise smooth with one or more jump discontinuities
at a finite set of times $0 < s_1 < s_2 < \ldots < s_{k} < T$ where there
is a fundamental change in the solution structure.
We also define $s_0 = 0$ and $s_{k+1} = T$.
For illustration, suppose that $\C{U}(t) =$
$\{ \m{v}\in\mathbb{R}^m : \g{\alpha}(t) \le \m{v} \le \g{\beta}(t) \}$,
where $\g{\alpha}$ and $\g{\beta} : [0,T] \rightarrow \mathbb{R}^m$,
and $\m{f}(\m{x}, \m{u}) = \m{g}(\m{x}) + \m{B}(\m{x})\m{u}$ with
$\m{B} : \mathbb{R}^n \rightarrow \mathbb{R}^{n \times m}$ and
$\m{g}:\mathbb{R}^n \rightarrow \mathbb{R}^n$.
The switching function, the coefficient of $\m{u}$ in the Hamiltonian,
is given by $\C{S}(t) = \m{p}(t) \m{B}(\m{x}(t))$.
Under the assumptions of the
Pontryagin minimum principle, an optimal solution
of (\ref{CP}) has the property that
\[
\begin{array}{lcl}
u_i (t) &=& \alpha_i (t) \mbox{ if } \C{S}_i(t) > 0, \\
u_i (t) &=& \beta_i (t) \mbox{ if } \C{S}_i(t) < 0, \mbox{ and} \\
u_i(t) &\in& [\alpha_i (t), \beta_i (t)] \mbox{ otherwise,}
\end{array}
\]
for each index $i \in [1, m]$ and for almost every $t \in [0, T]$.
On intervals where $\C{S}_i$ is either strictly positive or
strictly negative, $u_i$ is uniquely determined from the
minimum principle, and the control is said to be bang-bang.
If for some $i$, $\C{S}_i$ vanishes on an interval $[\sigma, \tau]$,
then the problem is said to be singular, and on this singular interval,
the first-order optimality conditions provide no information
concerning $u_i$ except that it satisfies the control constraints.

On any singular interval $(s_j, s_{j+1})$, not only does a component $\C{S}_i$
of the switching function vanish, but also \cite{Schattler12} the derivatives of
$\C{S}_i$ vanish, assuming they exist.
If the singularity has finite order, then
after equating derivatives of the switching function to zero,
we eventually obtain a relationship of the form
$u_i(t) = \phi_{ij} (\m{x}(t), \m{p}(t), t)$ for all $t \in (s_j, s_{j+1})$.
In vector notation, this relation can be expressed
$\m{u}(t) = \g{\phi}_{j} (\m{x}(t), \m{p}(t), t)$.
In many cases, it is possible to further simply this to
$\m{u}(t) = \g{\phi}_{j} (\m{x}(t), t)$, where there is no dependence of
the control on $\m{p}$.

In the Switch Point Algorithm (SPA), two separate cases are considered:
\smallskip
\begin{itemize}
\item[Case 1.]
For every $j$, $\g{\phi}_{j}$ is independent of $\m{p}$.
\item[Case 2.]
For some $j$, $\g{\phi}_{j}$ depends on $\m{p}$.
\end{itemize}
\smallskip
In a nutshell, the Switch Point Algorithm is based on the following
observations:
In Case~1, the control has the feedback form
$\m{u}(t) = \g{\phi}_j(\m{x}(t), t)$ for $t \in (s_j, s_{j+1})$,
$0 \le j \le k$.
For any given choice of the switching points, the solution of
the state dynamics (assuming a solution exists) yields a
value for the objective $C(\m{x}(T))$.
Hence, we can also think of the objective as a function
$C(\m{s})$ depending on $\m{s}$.
In Case~2, where the control also depends on $\m{p}$, we could
(assuming a solution exists),
integrate forward in time the coupled state and costate dynamics
from any given initial condition $\m{p}(0) = \m{p}_0$
with the control given by
$\m{u}(t) = \g{\phi}_j(\m{x}(t), \m{p}(t), t)$ on $(s_j, s_{j+1})$,
$0 \le j \le k$,
to obtain a value for the objective.
The objective would then be denoted $C(\m{s}, \m{p}_0)$
since $\m{x}(T)$ depends on both $\m{s}$ and $\m{p}_0$.

Suppose that $\m{s} = \m{s}^*$ corresponds to the switching points
for a solution of (\ref{CP}) and $\m{p}(0) = \m{p}_0^*$ is the
associated initial costate.
In many applications, one finds that
$\m{u}(t) = \g{\phi}_j(\m{x}(t), t)$ or
$\m{u}(t) = \g{\phi}_j(\m{x}(t), \m{p}(t), t)$ remain feasible in (\ref{CP})
for $\m{s}$ in a neighborhood of $\m{s}^*$ and for $\m{p}(0)$
in a neighborhood of $\m{p}_0^*$.
Moreover, $C(\m{s})$ or $C(\m{s}, \m{p}_0)$ achieves a local minimum at
$\m{s}^*$ or $(\m{s}^*, \m{p}_0^*)$.
Therefore, at least locally, we could replace (\ref{CP}) by the
problem of minimizing the objective over $\m{s}$ and $\m{p}_0$.

We briefly review previous numerical approaches for singular control problems.
Since the literature in the area is huge, our goal is to mostly highlight
historical trends.
One of the first approaches for singular control problems was
what Jacobson, Gershwin, and Lele \cite{Jacobson70}
called the $\epsilon$~-~$\alpha(\cdot)$ algorithm,
although today it would be called the proximal point method.
The idea was to make a singular problem nonsingular by adding a
strongly convex quadratic term to the objective.
The new objective is
\[
C(\m{x}(T)) + \frac{\epsilon_k}{2} \int_0^T \|\m{u}(t) - \m{u}_k (t)\|^2 ~ dt,
\]
and the solution of the problem with the modified objective yields
$\m{u}_{k+1}$.
Jacobson denoted the approximating control sequence by $\alpha_k$,
so the scheme was referred to as the $\epsilon$~-~$\alpha(\cdot)$ algorithm.
The choice of $\epsilon_k$ is a delicate issue; if it is too small,
then $\m{u}_{k+1}$ can oscillate wildly, but
if it is chosen just right, good approximations to solutions of
singular control problems have been obtained.

In a different approach, Anderson \cite{Anderson1972} considers a problem
where the control starts out nonsingular, and then changes to singular
at switch time $s_1$.
It is proposed to make a guess for the initial costate $\m{p}(0)$ and
then adjust it until the junction conditions at the boundary between
the singular and nonsingular control are satisfied.
Then further adjustments to the initial costate are made to satisfy the
terminal conditions on the costate.
If these further adjustments cause the junction conditions to be
violated unacceptably, then the entire process would be repeated,
first modifying the initial costate to update the switching point and
to satisfy the junction conditions, and
then to further modify the initial costate to satisfy the terminal
conditions.
Aly \cite{Aly1978} also considers the method of Anderson,
but with more details.
Maurer \cite{Maurer1976} uses multiple shooting techniques to
satisfy junction conditions and boundary conditions in singular
control problems.
In general, choosing switching points and adjusting the initial
costate condition to satisfy the junction and terminal conditions
could be challenging.

Papers closer in spirit to the Switch Point Algorithm include more recent
works of Maurer {\it et al.} \cite{Maurer05} and Vossen \cite{Vossen2010}.
In both cases, the authors consider general boundary conditions of
the form $\phi(\m{x}(0), \m{x}(T), T) = \m{0}$,
where $\phi : \mathbb{R}^{2n+1} \rightarrow \mathbb{R}^r$,
$0 \le r \le 2n$.
In \cite{Maurer05} the authors focus on bang-bang control problems, while
Vossen considers singular problems where the control
has the feedback form of Case~1.
In both papers, the authors view the objective as a function of the
switching points, and the optimization problem becomes a finite dimensional
minimization over the switching points and the initial state subject to the
boundary condition.
Vossen in his dissertation \cite{VossenThesis}
and in \cite[Prop~4.12]{Vossen_online} shows that when the switching
points correspond to the switching points of a control for the continuous
problem which satisfies the minimum principle, then the first-order optimality
conditions are satisfied for the finite dimensional optimization problem.
This provides a rigorous justification for the strategy of replacing
the original infinite dimensional control problem by a finite dimensional
optimization over the switching points and the boundary values of the state.

A fundamental difference between our approach and the approaches in
\cite{Maurer05} and \cite{Vossen2010} is that in the earlier work,
the derivative of the objective is expressed in terms
of the partial derivative of each state variables with respect to
each switching point, where this matrix of partial derivatives is
obtained by a forward propagation using the system dynamics.
We circumvent the evaluation of the matrix of partial
derivatives by using the costate equation to directly compute the
partial derivative of the cost with respect to all the switching points;
there is one forward integration of the state equation and one backward
integration of the costate equation to obtain the partial derivatives
of the objective with respect to all the switching points at the same time.
In a sense, our approach is a generalization of
\cite[Thm. 2]{bang_bang} which considers purely bang-bang controls.
One benefit associated with
the computation of the matrix of partial derivatives of each state
with respect to each switch point is that with marginal additional work,
second-order optimality conditions can be checked.

When the control has the feedback form of Case~1 and the dynamics are
affine in the control, our formula for the objective derivative
reduces to the product between the switching function and the jump
in the control at the switching point.
Consequently, at optimality in the finite dimensional problem,
the switching function should vanish.
The vanishing of the switching function at the switching point
is a classic necessary optimality condition
\cite{Agrachev02, OsmolovskiiMaurer07, Pontryagin62}.
Our formula, however, applies to arbitrary, not necessarily
optimal controls, and hence it can be combined with a nonlinear programming
algorithm to solve the control problem (\ref{CP}).

The methods discussed so far are known as indirect methods;
they are indirect in the sense that steps of the algorithm
employ information gleaned from the first-order optimality conditions.
In a direct method, the original continuous-in-time problem is
discretized using finite elements, finite differences, or collocation
to obtain a finite dimensional problem which is solved by a nonlinear
optimization code.
These discrete problems can be difficult due to both ill conditioning
and to discontinuities in the optimal control at the switching points.
If the location of the switching points were known, then high order
approximations are possible using orthogonal collocation techniques
\cite{HagerHouRao19}.
But in general, the switching points are not known, and mesh refinement
techniques
\cite{DarbyHagerRao11, DarbyHagerRao10, LiuHagerRao15, LiuHagerRao18,
PattersonHagerRao15} can lead to highly refined meshes
in a neighborhood of discontinuities, which can lead to a large dimension
for the discrete problem in order to achieve a specified error tolerance.

Betts briefly touched on a hybrid direct/indirect approach to singular
control in \cite[Sect.~4.14.1]{Betts2010} where the switching points between
singular and nonsingular regions are introduced as variables in the
discrete problem, a junction condition is imposed at the switching points,
and the form of the control in the singular region is explicitly imposed.
A relatively accurate solution of the Goddard rocket problem
\cite{Bryson75} was obtained.

In a series of papers,
see \cite{Biegler20, Biegler19b, Biegler16, Biegler19a} and
the reference therein, Biegler and colleagues develop approaches to
singular control problems that combine ideas from both direct and indirect
schemes.
In \cite{Biegler16}, 
the algorithm utilizes an inner problem where
the mesh and the approximations to the switching points are fixed, and the
discretized control problem is solved by a nonlinear optimization code
such as IPOPT \cite{Wachter06}.
Then an outer problem is formulated where the finite element mesh is
modified so as to reduce errors in the dynamics or make the switching
function and its derivative closer to zero in the singular region.
In \cite{Biegler19a} more sophisticated rules are developed for moving grid
points or either inserting or deleting grid points
by monitoring errors or checking for spikes.
In \cite{Biegler19b}, the inner and outer problems are also combined
to form a single nonlinear program that is optimized.
In \cite{Biegler20} the direct method is mostly used to obtain a starting
guess for the indirect phase of the algorithm.
In the indirect phase, special structure is imposed on the control to
reduce or eliminate wild oscillation, and conditions are imposed to
encourage the vanishing of the switching function in the singular region
and the constancy of the Hamiltonian.

In comparing the Switch Point Algorithm to the existing literature,
the necessary optimality conditions are not imposed except for our assumption,
in the current version of the algorithm, 
that the control in the singular region has been expressed as a
function of the state and/or costate.
Note that the papers of Biegler and colleagues do not make this assumption;
instead the vanishing of the switching function in the singular region
is a constraint in their algorithms.
SPA focuses on minimizing the objective with respect to the switching points;
presumably, the necessary optimality conditions will be satisfied at a
minimizer of the objective, but these conditions are not imposed
on the iterates.

After solving SPA's finite dimensional optimization problem
associated with a bang-bang or singular control problem, one may wish
to check the optimality of the computed finite dimensional solution
in the original continuous control problem (\ref{CP});
to check whether a switching point was missed
in the finite dimensional optimization problem, one could return to
the original continuous control problem, and test whether the
minimum principle holds at other points in the domain,
not just at the switching points, by performing an accurate integration
of the differential equations between the switching points.
To check the local optimality of the computed solution of the
finite dimensional problem, one could test the second-order sufficient
optimality conditions.
As noted in \cite[Thm.~3.1]{Maurer05} for bang-bang control,
satisfaction of the second-order sufficient optimality conditions in the
finite dimensional problem implies that 
the associated solution of the continuous control problem is a strict minimum.
Second-order sufficient optimality conditions for bang-singular controls 
are developed in \cite{Vossen2010}.

The paper is organized as follows:
In Section~\ref{dscase1}, we explain how to compute the derivative of
$C$ with respect to $\m{s}$ (Case~1), while Section~\ref{dscase2} deals with
the derivative of $C$ with respect to both $\m{s}$ and $\m{p}_0$
(Case~2).
Section~\ref{free} considers free terminal time problems, and obtains the
derivative of the objective with respect to the final time $T$.
The efficient application of SPA
requires a good guess for the structure of an optimal control.
In Section~\ref{starting} we explain one approach for generating a
starting guess using total variation regularization, which has been
effective in image reconstruction \cite{ROF} at replacing blurry edges
with sharp edges.
Finally, Section~\ref{examples} provides some comparisons with
other methods from the literature using test problems with known solutions.

{\bf Notation and Terminology.}
By a valid choice of the switch points,
we mean that $\m{s}$ satisfies the relations
$0 = s_0 < s_1 < s_2 < \ldots < s_k < s_{k+1} = T$.
Throughout the paper, $\| \cdot \|$ is any norm on $\mathbb{R}^n$.
The ball with center $\m{c} \in \mathbb{R}^n$ and radius $\rho$ is given by
$\C{B}_\rho (\m{c}) = \{ \m{x} \in \mathbb{R}^n : \|\m{x} - \m{c}\| \le \rho\}$.
The expression $\C{O}(\Delta s)$ denotes a quantity that is bounded
in absolute value by $c |\Delta s|$, where $c$ is a constant
that is independent of $\Delta s$.
Given $\m{x}$ and $\m{y} \in \mathbb{R}^n$, we let $[\m{x}, \m{y}]$ denote
the line segment connecting $\m{x}$ and $\m{y}$.
In other words,
\[
[\m{x}, \m{y}] = \{\m{x} + \alpha (\m{y}-\m{x}) : \alpha \in [0, 1]\}.
\]
The Jacobian of $\m{f}(\m{x}, \m{u})$ with respect to $\m{x}$ is denoted
$\nabla_x \m{f}(\m{x}, \m{u})$; its $(i, j)$ element is
$\partial f_i (\m{x}, \m{u})/ \partial x_j$.
For a real-valued function such as $C$, the gradient
$\nabla C(\m{x})$ is a row vector.
The costate and the generalized costate (introduced in Section~\ref{dscase2})
are both row vectors, while all other
vectors in the paper are column vectors.
The $L^2$ inner product on $[0, T]$ is denoted $\langle \cdot , \cdot \rangle$.
%
\section{Objective Derivative in Case~1}
\label{dscase1}
In Case~1, it is assumed that the control has the form
$\m{u}(t) = \g{\phi}_j(\m{x}(t), t)$ for all $t \in (s_j, s_{j+1})$,
$0 \le j \le k$.
With this substitution, the dynamics in (\ref{CP}) has the form
$\m{\dot{x}}(t) = \m{f}_j (\m{x}(t), t)$ on $(s_j, s_{j+1})$ where
$\m{f}_j (\m{x}, t) = \m{f}(\m{x}, \g{\phi}_j (\m{x}, t))$.
Note that $\m{f}_j$ is viewed as a mapping from
$\mathbb{R}^n \times [0, T]$ to $\mathbb{R}^n$.
The objective is $C(\m{x}(T))$ where $\m{x}$ is the solution to an
initial value problem of the form
\begin{equation}\label{IVP}
\m{\dot{x}}(t) = \m{F} (\m{x}(t), t), \quad
\m{F}(\m{x}, t) = \m{f}_j (\m{x}, t)
\mbox{ for } t \in (s_j, s_{j+1}), \quad
\m{x}(0) = \m{x}_0,
\end{equation}
$0 \le j \le k$.
In the Switch Point Algorithm, the goal is to minimize the objective
value over the choice of the switching points $s_1, s_2, \ldots , s_k$.
This minimization can be done more efficiently if the gradient of the
objective with respect to the switching points is known since superlinearly
convergent algorithms such as the conjugate gradient method or
a quasi-Newton method could be applied.
In Theorem~\ref{spa1}, a formula is derived for the gradient of $C$
with respect to $\m{s}$.
The following three preliminary results are used in the analysis.
\smallskip
\begin{lemma}\label{Lip}
If $\m{x} : [\sigma_1, \sigma_2] \rightarrow \mathbb{R}^n$ is Lipschitz
continuous, then so is $\|\m{x}(\cdot)\|$ and
\begin{equation}\label{deriv0}
\frac{d}{dt} \|\m{x}(t)\| \le \|\m{\dot{x}}(t)\|
\end{equation}
for almost every $t \in [\sigma_1, \sigma_2]$.
\end{lemma}
\begin{proof}
For any $s, t \in [\sigma_1, \sigma_2]$, the triangle inequality gives
\[
\|\m{x}(s)\| = \|\m{x}(s) - \m{x}(t) + \m{x}(t)\| \le
\|\m{x}(s) - \m{x}(t)\| + \|\m{x}(t)\|.
\]
Rearrange this inequality to obtain
\begin{equation}\label{deriv1}
\|\m{x}(s)\| - \|\m{x}(t)\| \le \|\m{x}(s) - \m{x}(t)\|.
\end{equation}
Interchanging $s$ and $t$ in (\ref{deriv1}) yields
$\|\m{x}(t)\| - \|\m{x}(s)\| \le \|\m{x}(s) - \m{x}(t)\|$.
Hence, the absolute value of the difference
$\|\m{x}(s)\| - \|\m{x}(t)\|$ is bounded by
$\|\m{x}(t) - \m{x}(s)\|$.
Since $\m{x}(\cdot)$ is Lipschitz continuous, then so is
$\|\m{x}(\cdot)\|$.
It follows by Rademacher's Theorem that both $\m{x}(\cdot)$ and
$\|\m{x}(\cdot)\|$ are differentiable almost everywhere.
Suppose $t \in (\sigma_1, \sigma_2)$ is a point of differentiability
for both $\m{x}(\cdot)$ and $\|\m{x}(\cdot)\|$, and take $s = t + \Delta t$
in (\ref{deriv1}).
Dividing the resulting inequality by $\Delta t$ and letting $\Delta t$
tend to zero yields (\ref{deriv0}).
\end{proof}
\smallskip

The following result can be deduced from Gronwall's inequality.
\begin{lemma} \label{gronwall}
If $w : [0, T] \rightarrow \mathbb{R}$ is absolutely continuous
and for some nonnegative scalars $a$ and $b$,
\[
\dot{w}(t) \le a w(t) + b, \quad \mbox{for almost every } t \in [0, T],
\]
then for all $t \in [0,T]$,
we have
\begin{equation}\label{gronwallbound}
w(t) \le e^{at} (w(0) + bt) .
\end{equation}
\end{lemma}

The following Lipschitz result is deduced from Lemmas~\ref{Lip} and
\ref{gronwall}.
\begin{corollary} \label{IC}
Suppose that $\m{x}$ is an absolutely continuous solution
of $(\ref{IVP})$ and $\m{y}$ has the
same dynamics but a different initial condition:
\begin{equation}\label{y-IVP}
\m{\dot{y}}(t) = \m{F}(\m{y} (t), t), \quad \m{y}(0) = \m{y}_0.
\end{equation}
If for some $\rho > 0$ and $L \ge 0$, independent of $j$,
$\m{f}_j(\g{\chi}, t)$, $0 \le j \le k$, is continuous with respect to
$\g{\chi}$ and $t$ and Lipschitz continuous
in $\g{\chi}$, with Lipschitz constant $L$, on the tube
\begin{equation}\label{tube}
\{ (\g{\chi}, t) :
t \in [s_j, s_{j+1}] \mbox{ and }
\g{\chi} \in \C{B}_\rho(\m{x}(t)) \},
\end{equation}
then for any $\m{y}_0 \in \mathbb{R}^n$ which satisfies
$e^{LT} \|\m{x}_0 - \m{y}_0\| \le \rho$, the initial value problem
$(\ref{y-IVP})$ has a solution on $[0, T]$, and we have
\begin{equation}\label{errorw}
\|\m{y}(t) - \m{x}(t)\| \le e^{Lt} \|\m{y}_0 - \m{x}_0\|
\quad \mbox{for all } t \in [0, T].
\end{equation}
\end{corollary}
\begin{proof}
For $t = 0$ and $j = 0$, $\m{y}_0$ is in the interior of a face of the
tube (\ref{tube}).
Due to the Lipschitz continuity of $\m{F}(\cdot, t)$,
a solution to (\ref{y-IVP}) exists for near $t = 0$.
Define $w(t) = \|\m{y}(t) - \m{x}(t)\|$,
subtract the differential equations (\ref{IVP}) and (\ref{y-IVP}), and
take the norm of each side to obtain
\begin{equation}\label{normdot}
\| \m{\dot{x}}(t) - \m{\dot{y}}(t)\| \le L w(t)
\end{equation}
for any $t$ where $\m{y}$ satisfies (\ref{y-IVP}) and
$\m{y}(t)$ lies within the tube around $\m{x}(t)$ where
$\m{F}(\cdot, t)$ satisfies the Lipschitz continuity property.
Any solution $\m{x}$ of (\ref{IVP}) and $\m{y}$ of (\ref{y-IVP})
is Lipschitz continuous since their derivatives are bounded.
Lemma~\ref{Lip} and equation (\ref{normdot}) imply that
$\dot{w}(t) \le L w(t)$ for $t$ near 0.
Hence, for $t$ near 0, the bound (\ref{gronwallbound}) of Lemma~\ref{gronwall}
with $b = 0$ yields
\[
w(t) \le e^{Lt} w(0) = e^{Lt} \|\m{x}_0 - \m{y}_0\| < \rho \quad
\mbox{when } t < T.
\]
For $t < T$, $\m{y}(t)$ continues to lie in the interior of the
tube where $\m{F}(\cdot, t)$ satisfies the Lipschitz property,
which leads to (\ref{errorw}).
\end{proof}

In order to analyze the objective derivative with respect to
the switch points,
we need to make a regularity assumption concerning the functions $\m{f}_j$ in
the initial value problem (\ref{IVP}) to ensure the stability and uniqueness of
solutions to (\ref{IVP}) as the switching points are perturbed around a
given $\m{s}$.
\smallskip

{\bf State Regularity.}
Let $\m{x}$ denote an absolutely continuous solution to (\ref{IVP}).
It is assumed that there exist constants $\rho > 0$,
$s_j^- \in (s_{j-1}, s_j)$, and $s_j^+ \in (s_j , s_{j+1})$,
$1 \le j \le k$, such that
$\m{f}_j$ is continuously differentiable on the tube
\[
\C{T}_j = \{ (\g{\chi}, t) : t \in [s_j^-, s_{j+1}^+]
\mbox{ and } \g{\chi} \in \C{B}_\rho(\m{x}(t)) \},
\quad 0 \le j \le k,
\]
where $s_0^- = 0$ and $s_{k+1}^+ = T$.
Moreover, $\m{f}_j(\g{\chi}, t)$ is Lipschitz continuously differentiable
with respect to $\g{\chi}$ on $\C{T}_j$,
uniformly in $j$ and $t \in [s_j^-, s_j^+]$.
\smallskip

\begin{theorem} \label{spa1}
Suppose that $\m{s}$ is a valid choice for the switching points,
that the State Regularity property holds, and
$C$ is Lipschitz continuously differentiable in a neighborhood of $\m{x}(T)$.
Then for $j = 1, 2, \ldots , k$, $C(\m{s})$ is differentiable with
respect to $s_j$ and
\begin{equation} \label{case1d}
\frac{\partial C}{\partial s_j} (\m{s}) =
H_{j-1} (\m{x}(s_j), \m{p}(s_j), s_j)
- H_{j}(\m{x}(s_j), \m{p}(s_j), s_j),
\end{equation}
where $H_j (\m{x}, \m{p}, t) = \m{p}\m{f}_j(\m{x}, t)$,
and the row vector $\m{p}: [0, T] \rightarrow \mathbb{R}^n$
is the solution to the linear differential equation
\begin{equation}\label{p-case1}
\m{\dot{p}}(t) = - \m{p}(t) \nabla_x \m{F}(\m{x}(t), t),
\quad t \in [0, T],
\quad
\m{p}(T) = \nabla C(\m{x}(T)) .
\end{equation}
\end{theorem}
\smallskip

\begin{remark}\label{notequal}
The formula $(\ref{case1d})$ for the derivative of
the objective $C$ with respect to a switching point generalizes the
result derived in \cite[Thm. 2]{bang_bang} for a bang-bang control.
Note that equation $(\ref{p-case1})$ differs from the
costate equation in $(\ref{p})$ since the Jacobian of $\m{f}$ appears in
$(\ref{p})$, while the Jacobian of $\m{F}$ appears in $(\ref{p-case1})$.
However, when the control $\m{u}$ enters the Hamiltonian linearly,
$\nabla_x \m{F}$ approaches $\nabla_x \m{f}$ as the switch points approach
the switch points for an optimal solution of the control problem,
and the solution $\m{p}$ of $(\ref{p-case1})$ approaches the costate of
the optimal solution.
\end{remark}
\smallskip

\begin{proof}
Let $s$ denote the switching point $s_j$.
To compute the derivative of the cost with respect to $s$, we need to
compare the cost associated with $s$ to the cost gotten when $s$ is changed
to $s + \Delta s$.
Let $\m{x}$ and $\m{y}$ denote the solutions of (\ref{IVP}) associated with
the original and the perturbed switching points respectively.
The dynamics associated with these two solutions are identical except for
the time interval $[s, s+\Delta s]$.
With the definitions $\m{F}_0 = \m{f}_{j-1}$ and $\m{F}_1 = \m{f}_j$,
the dynamics for $t \in [s, s+\Delta s]$ and
for $\Delta s$ sufficiently small are
\begin{equation} \label{sdynamic}
\m{\dot{x}}(t) = \m{F}_1 (\m{x}(t), t) \quad \mbox{and} \quad
\m{\dot{y}}(t) = \m{F}_0 (\m{y}(t), t) .
\end{equation}
The objective value associated with the switching point $s$ is
$C(\m{x}(T))$, where $\m{x}$ is the solution of the initial value problem
(\ref{IVP}).
The objective value associated with the switching point $s + \Delta s$ is
$C(\m{y}(T))$ where $\m{y}$ has the same dynamics as $\m{x}$ except on the
interval $[s, s+\Delta s]$.
To evaluate the derivative of the objective with respect to $s$, we need to
form the ratio $[C(\m{y}(T)) - C(\m{x}(T))]/\Delta s$, and then let
$\Delta s$ tend to zero.

Since the dynamics for $\m{x}$ and $\m{y}$ are identical except for the
interval $[s, s+ \Delta s]$,
$\m{x}(t) = \m{y}(t)$ for $t \le s$.
Let $\m{x}_s$ denote $\m{x}(s) = \m{y}(s)$. 
Expanding $\m{x}$ and $\m{y}$ in Taylor series around $\m{x}_s$ yields
\begin{eqnarray}
\m{x}(s + \Delta s) &=& \m{x}_s + (\Delta s) \m{F}_1 (\m{x}_s, s) +
\C{O}((\Delta s)^2), \label{x-Taylor} \\
\m{y}(s + \Delta s) &=& \m{x}_s + (\Delta s) \m{F}_0 (\m{x}_s, s) +
\C{O}((\Delta s)^2), \label{y-Taylor}
\end{eqnarray}
where the remainder term is $\C{O}((\Delta s)^2)$ due to the
State Regularity property, which ensures the Lipschitz continuous
differentiability of $\m{F}_0$ and $\m{F}_1$.
For $t > s + \Delta s$, $\m{\dot{x}}(t) = \m{F} (\m{x}(t), t)$ and
$\m{\dot{y}}(t) = \m{F} (\m{y}(t), t)$ since
the dynamics of $\m{x}$ and $\m{y}$ only differ on $[s, s+\Delta s]$.
Apply Corollary~\ref{IC} to the interval $[s+\Delta s, T]$.
Based on the expressions (\ref{x-Taylor}) and (\ref{y-Taylor}) for
the values of $\m{x}$ and $\m{y}$ at $s + \Delta s$,
it follows from Corollary~\ref{IC} that
\begin{equation}\label{ychange}
\|\m{y}(t) - \m{x}(t)\| = \C{O}(\Delta s), \quad \mbox{for all }
t \in [s+\Delta s, T].
\end{equation}

Let $\m{z} : [s+\Delta s, T] \rightarrow \mathbb{R}^n$ be the solution
to the linear differential equation
\begin{equation}\label{z-def}
\m{\dot{z}}(t) = \nabla_x \m{F} (\m{x}(t), t) \m{z}(t), \quad
\m{z}(s+\Delta s) = \Delta s [\m{F}_0(\m{x}_s, s) -\m{F}_1(\m{x}_s, s)].
\end{equation}
By assumption, $\m{x}$ is absolutely continuous and hence bounded;
consequently, there exists a scalar $a$ such that
$\|\nabla_x \m{F} (\g{\chi}, t)\| \le a$ for all
$t \in [s, T]$ and $\g{\chi} \in \C{B}_\rho (\m{x}(t))$.
Take the norm of the differential equation for $\m{z}$ and apply Lemma~\ref{Lip}
to obtain
\[
\frac{d \|\m{z}(t)\|}{dt} \le a \|\m{z}(t)\| \quad \mbox{for all }
t \in [s+\Delta s, T].
\]
By Lemma~\ref{gronwall} with $b = 0$ and the specified initial condition
in (\ref{z-def}), it follows that
\begin{equation}\label{z-bound}
\|\m{z}(t)\| = \C{O}(\Delta s) \quad \mbox{for all } t \in [s+\Delta s, T].
\end{equation}

Define $\g{\delta}(t) = \m{y}(t) - \m{x}(t) - \m{z}(t)$ for every
$t \in [s + \Delta s, T]$ and
\[
{\m{x}}(\alpha,t) = \m{x}(t) + \alpha(\m{y} (t) - \m{x}(t)).
\]
Differentiating $\g{\delta}$ and utilizing a Taylor expansion with
integral remainder term, we obtain for all $t \in [s + \Delta s, T]$,
\begin{eqnarray}
\g{\dot{\delta}}(t) &=& \m{\dot{y}}(t) - \m{\dot{x}}(t) - \m{\dot{z}}(t) =
\m{F}(\m{y}(t), t) - \m{F} (\m{x}(t), t) -
\nabla_x \m{F} (\m{x}(t), t) \m{z}(t) \label{deltadot} \\
&=& \left( \int_0^1 \nabla_x \m{F} ({\m{x}}(\alpha,t), t) \; d\alpha
\right) (\m{y}(t) - \m{x}(t))
- \left( \int_0^1 \nabla_x \m{F} (\m{x}(t), t) \; d\alpha \right)
\m{z}(t) \nonumber \\
&=& \left( \int_0^1 [\nabla_x \m{F} ({\m{x}}(\alpha,t), t) -
\nabla_x \m{F} (\m{x}(t), t) ] \; d\alpha \right) \m{z}(t)
+ \left( \int_0^1 \nabla_x \m{F} ({\m{x}}(\alpha,t), t) \; d\alpha \right)
\g{\delta}(t) . \nonumber
\end{eqnarray}
Take $\Delta s$ in (\ref{ychange}) small enough that $\m{y}(t)$ lies in the
tube around $\m{x}(t)$ where $\nabla_x \m{F}$ is Lipschitz continuous.
By the definition of $a$, we have
$\|\nabla_x \m{F} ({\m{x}}(\alpha,t), t)\| \le a$.
If $L$ is the Lipschitz constant for $\nabla_x \m{F}$, then
by (\ref{ychange}), we have
\[
\|\nabla_x \m{F} ({\m{x}}(\alpha,t), t) - \nabla_x \m{F} (\m{x}(t), t)\|
\le \alpha L \|\m{y}(t) - \m{x}(t)\| = \C{O}(\Delta s) .
\]
Hence, taking the norm of each side of (\ref{deltadot})
and utilizing Lemma~\ref{Lip} on the left side, and the triangle inequality
and the bound (\ref{z-bound}) on the right side yields
\begin{equation}\label{gronwalldelta}
\frac{d \|\g{\delta}(t)\|}{dt} \le
a \|\g{\delta}(t)\| + \C{O}((\Delta s)^2) .
\end{equation}
By (\ref{x-Taylor}), (\ref{y-Taylor}), and (\ref{z-def}),
we have
\[
\g{\delta}(s + \Delta s) =
\m{y}(s + \Delta s) - \m{x}(s + \Delta s) - \m{z}(s + \Delta s) =
\C{O}((\Delta s)^2).
\]
Consequently, by (\ref{gronwalldelta}) and Lemma~\ref{gronwall}, we deduce that
\begin{equation}\label{deltabound}
\|\g{\delta}(t)\| = \C{O}((\Delta s)^2) \quad \mbox{for all }
t \in [s+\Delta s, T].
\end{equation}

If $\m{p}$ is the solution of (\ref{p-case1}) and $\m{z}$ is the solution
of (\ref{z-def}), then we have
\begin{eqnarray}
0 &=& \int_{s+\Delta s}^T \m{p}(t)
\bigg[\nabla_x \m{F}(\m{x}(t), t)\m{z}(t) - \m{\dot{z}}(t)\bigg] \; dt
\nonumber \\
&=& \int_{s+\Delta s}^T \bigg[ \m{p}(t)
\nabla_x \m{F}(\m{x}(t), t) + \m{\dot{\m{p}}}(t)\bigg] \m{z}(t) \; dt
- \m{p}(T)\m{z}(T) +\m{p}(s+\Delta s) \m{z}(s+\Delta s)
\nonumber  \\[.05in]
&=& - \m{p}(T)\m{z}(T) +\m{p}(s+\Delta s) \m{z}(s+\Delta s)
\nonumber  \\[.05in]
&=& 
\Delta s \m{p}(s+\Delta s) (\m{F}_0(\m{x}_s, s) - \m{F}_1(\m{x}_s, s))
- \nabla C(\m{x}(T))\m{z}(T). \label{BC}
\end{eqnarray}

Since $C$ is Lipschitz continuously differentiable in a neighborhood
of $\m{x}(T)$, the difference between the perturbed objective and the original
objective can be expressed
\begin{equation}\label{deltaC}
C(\m{y}(T)) - C(\m{x}(T)) = \nabla C(\m{x}_{\Delta})(\m{y}(T) - \m{x}(T))
\end{equation}
where $\m{x}_\Delta \in [\m{y}(T), \m{x}(T)]$;
that is, $\m{x}_\Delta$ is a point on the line segment connecting
$\m{y}(T)$ and $\m{x}(T)$.
Since the distance between $\m{x} (t)$ and $\m{y}(t)$ is $\C{O}(\Delta s)$ by
(\ref{ychange}), the distance between $\m{x}_\Delta$ and $\m{x}(T)$ is
$\C{O}(\Delta s)$.
Add the right side of (\ref{BC}) to the right side of (\ref{deltaC}) and
substitute
\[
\m{y}(T) - \m{x}(T) =
\m{y}(T) - \m{x}(T) -\m{z}(T) + \m{z}(T) = \g{\delta}(T) + \m{z}(T)
\]
to obtain
\begin{eqnarray}
C(\m{y}(T)) - C(\m{x}(T)) &=& \nabla C(\m{x}_{\Delta})\g{\delta}(T) +
[\nabla C(\m{x}_{\Delta})- \nabla C(\m{x}(T)]\m{z}(T) \nonumber \\
&& \quad \quad + \Delta s \m{p}(s+\Delta s)
[\m{F}_0(\m{x}_s, s) - \m{F}_1(\m{x}_s, s)].
\label{final}
\end{eqnarray}
By (\ref{deltabound}), $\|\g{\delta}(T)\| = \C{O}((\Delta s)^2)$ so
$|\nabla C(\m{x}_{\Delta})\g{\delta}(T)| = \C{O}((\Delta s)^2)$.
Since $C$ is Lipschitz continuously differentiable at $\m{x}(T)$,
the distance from $\m{x}_{\Delta}$ to $\m{x}(T)$ is at most $\C{O}(\Delta s)$
by (\ref{ychange}), and $\m{z}(T) = \C{O}(\Delta s)$ by (\ref{z-bound}),
we have
\[
\|[\nabla C(\m{x}_{\Delta})- \nabla C(\m{x}(T)]\m{z}(T) \| =
\C{O}((\Delta s)^2).
\]
Consequently, the first two terms on the right side of (\ref{final}) are
$\C{O}((\Delta s)^2)$.
Divide (\ref{final}) by $\Delta s$ and let $\Delta s$ tend
to zero to obtain
\[
\lim_{\Delta s \rightarrow 0}
\frac{C(\m{y}(T)) - C(\m{x}(T))}{\Delta s}
=  \m{p}(s) [\m{F}_0(\m{x}_s, s) - \m{F}_1(\m{x}_s, s)] ,
\]
which completes the proof since $\m{F}_0 = \m{f}_{j-1}$ and
$\m{F}_1 = \m{f}_j$.
\end{proof}
\section{Objective Derivative in Case~2}
\label{dscase2}
In Case~2, an optimal control which minimizes the Hamiltonian
also depends on $\m{p}$ in the singular region, so the control has the form
$\m{u}(t) = \g{\phi}_j(\m{x}(t), \m{p}(t), t)$ for $t \in (s_j, s_{j+1})$.
The functions $\m{f}_j$ and $\m{F}$ of Section~\ref{dscase1} now have the form
$\m{f}_j (\m{x}, \m{p}, t) = \m{f}(\m{x}, \g{\phi}_j (\m{x}, \m{p}, t))$
and $\m{F}(\m{x}, \m{p}, t) = \m{f}_j (\m{x}, \m{p}, t)$ for
$t \in (s_j, s_{j+1})$.
The Jacobian appearing in the costate dynamics is
$\nabla_x \m{f}(\m{x}, \m{u})$;
in the Switch Point Algorithm, we evaluate this dynamics at $\m{u} =$
$\g{\phi}_j(\m{x}, \m{p}, t)$.
Hence, analogous to definition given in Section~\ref{dscase1}, let us define
\[
\m{f}_{jx}(\m{x}, \m{p}, t) = \nabla_x \m{f}(\m{x}, \m{u})
\Big|_{\displaystyle \m{u} = \g{{\phi}}_j(\m{x}, \m{p}, t)}.
\]
Also, we define
\[
\m{F}_x (\m{x}, \m{p}, t) = \m{f}_{jx} (\m{x}, \m{p}, t),
\quad \mbox{for } t \in (s_j, s_{j+1}).
\]

In the Switch Point Algorithm, we consider the coupled system
\begin{equation} \label{coupled}
\m{\dot{\m{x}}}(t) = \m{F}(\m{x}(t), \m{p}(t), t), \quad
\m{\dot{\m{p}}}(t) = -\m{p}(t)\m{F}_x(\m{x}(t), \m{p}(t), t),
\end{equation}
where $(\m{x}(0), \m{p}(0)) = (\m{x}_0, \m{p}_0)$.
If $\m{u}^*$ is a local minimizer for the control problem (\ref{CP})
and $(\m{x}^*, \m{p}^*)$ are the associated state and costate, then
we could recover $\m{u}^*(t) =\g{\phi}_j(\m{x}^*(t), \m{p}^*(t), t)$,
$t \in (s_j, s_{j+1})$, by integrating the coupled system
(\ref{coupled}) forward in time starting from the initial condition
$(\m{x}(0), \m{p}(0)) = (\m{x}_0, \m{p}^*(0))$.
From this perspective, we can think of the objective
$C(\m{x}(T))$ as being a function $C(\m{s}, \m{p}_0)$ that depends
on both the switching points and the starting value $\m{p}_0$
for the costate (the starting condition for the state $\m{x}_0$ is given).
To solve the control problem, we will search for a local minimizer of
$C(\m{s}, \m{p}_0)$.
Again, to exploit superlinearly convergent optimization algorithms,
the derivatives of $C$ with respect to both $\m{s}$ and $\m{p}_0$
should be evaluated.

The derivative of the objective with respect to the switching points
in Case~2 is a corollary of Theorem~\ref{spa1}.
In this case, the $\m{x}$ that satisfies (\ref{IVP}) is identified with the
pair $(\m{x}, \m{p})$ which solves
the coupled system (\ref{coupled}).
The pair $(\m{x}, \m{p})$ might be viewed as a {\it generalized state} in the
sense that for a given starting value $\m{p}(0) = \m{p}_0$ and for a given
choice of the switch points, we can in principle integrate forward in time
the coupled system (\ref{coupled}) to evaluate the objective $C(\m{x}(T))$.
The generalized version of the State Regularity property,
which applies to the pair $(\m{x}, \m{p})$, is the following:
\smallskip

{\bf Generalized State Regularity.}
Let $(\m{x}, \m{p})$ denote an absolutely continuous solution
to (\ref{coupled}).
It is assumed that there exist constants $\rho > 0$,
$s_j^- \in (s_{j-1}, s_j)$, and $s_j^+ \in (s_j , s_{j+1})$,
$1 \le j \le k$, such that the pair
$(\m{f}_j, \m{f}_{jx})$ is continuously differentiable on the tube
\[
\C{T}_j = \{ (\g{\chi}, t) : t \in [s_j^-, s_{j+1}^+]
\mbox{ and } \g{\chi} \in \C{B}_\rho(\m{x}(t)) \},
\quad 0 \le j \le k,
\]
where $s_0^- = 0$ and $s_{k+1}^+ = T$.
Moreover, $(\m{f}_j(\g{\chi}, t), \m{f}_{jx}(\g{\chi}, t))$
is Lipschitz continuously differentiable
with respect to $\g{\chi}$ on $\C{T}_j$,
uniformly in $j$ and $t \in [s_j^-, s_j^+]$.
\smallskip

The {\it generalized costate} associated with the system (\ref{coupled})
is a row vector $\m{y} \in \mathbb{R}^{2n}$ whose first $n$ components are
denoted $\m{y}_1$ and whose second $n$ components are denoted $\m{y}_2$.
The generalized Hamiltonian is defined by
\[
\C{H}_j(\m{x}, \m{p}, \m{y}, t) =
\m{y}_1\m{f}_j(\m{x}, \m{p}, t)
- \m{p} \m{f}_{jx} (\m{x}, \m{p}, t)\m{y}_2\tr , \quad
0 \le j \le k .
\]
The generalized costate $\m{y}: [0, T] \rightarrow \mathbb{R}^{2n}$
is the solution of the linear system of differential equations
\begin{eqnarray}
\dot{\m{y}}_1 (t) &=& -\nabla_x \C{H}_j(\m{x}(t), \m{p}(t), \m{y}(t), t),
\quad
\m{y}_1(T) = \nabla C(\m{x}(T)) \label{case2-x} \\
\dot{\m{y}}_2 (t) &=& -\nabla_p \C{H}_j(\m{x}(t), \m{p}(t), \m{y}(t), t),
\quad \m{y}_2(T) = \m{0} \label{case2-p}
\end{eqnarray}
on $(s_{j}, s_{j+1})$ for $j = k,~ k-1,~ \ldots, ~0$.
If the Generalized State Regularity property holds, then
by Theorem~\ref{spa1}, we have
\begin{equation} \label{case2d}
\frac{\partial C}{\partial s_j} (\m{s}, \m{p}_0) =
\C{H}_{j-1}(\m{x}(s_j), \m{p}(s_j), \m{y}(s_j), s_j)
- \C{H}_j(\m{x}(s_j), \m{p}(s_j), \m{y}(s_j), s_j).
\end{equation}
for $j = 1,~ 2,~ \ldots,~ k$.
Note that the boundary condition for $\m{y}_2$ is $\m{y}_2(T) = \m{0}$
since there is no $\m{p}$ in the objective, the objective $C$ only depends
on $\m{x}(T)$.
\smallskip

\begin{remark} \label{nop}
If the formula $(\ref{case2d})$ is used to evaluate the derivative of
the objective with respect to a switch point in the case where
$\g{\phi}$ does not depend on $\m{p}$, then the formula $(\ref{case2d})$
reduces to the formula $(\ref{case1d})$.
In particular, when $\g{\phi}$ does not depend on $\m{p}$,
the dynamics for $\m{y}_2$ becomes
\[
\dot{\m{y}}_2 (t) = \m{y}_2 \m{F}_x(\m{x}, \g{\phi}_j (\m{x}(t), t))\tr, \quad
t \in (s_j, s_{j+1}), \quad j = k, k-1, \ldots, 0, \quad \m{y}_2(T) = \m{0}.
\]
Consequently, $\m{y}_2$ is the solution to a linear differential equation
with the initial condition $\m{y}_2(T) = \m{0}$.
The unique solution is $\m{y}_2 = \m{0}$, and when $\m{y}_2 = \m{0}$,
the equation $(\ref{case2-x})$ is the same as $(\ref{p-case1})$.
Thus $\m{y}_1 = \m{p}$ and the formula $(\ref{case2d})$ is the same as
$(\ref{case1d})$.
\end{remark}
\smallskip

Now let us consider the gradient of $C(\m{s}, \m{p}_0)$ with respect to
$\m{p}_0$.
Let $(\m{x}, \m{p})$ denote a solution of (\ref{coupled}) for a given
starting condition $\m{p}(0) = \m{p}_0$, and let $(\m{\bar{x}}, \m{\bar{p}})$
denote a solution corresponding to $\m{p}(0) = \m{\bar{p}}_0$.
Let $\m{y}$ denote the generalized costate associated with
$(\m{\bar{x}}, \m{\bar{p}})$.
Since $(\m{x},\m{p})$ is a solution of (\ref{coupled}), we have
\begin{equation}\label{int}
0 = \int_0^T \m{y}_1(t) [\m{F}(\m{x}(t), \m{p}(t), t) - \m{\dot{x}}(t)]
- [\m{p}(t)\m{F}_x(\m{x}(t), \m{p}(t), t) + \m{\dot{p}}(t)] \m{y}_2\tr(t) \; dt.
\end{equation}
The two derivative terms in (\ref{int}) are integrated by parts to obtain
\begin{equation}\label{boundary}
-[ \langle \m{y}_1, \m{\dot{x}}\rangle +
\langle \m{y}_2, \m{\dot{p}}\rangle] =
\langle \m{\dot{y}}_1, \m{x}\rangle +
\langle \m{\dot{y}}_2, \m{p}\rangle
+ \m{y}_1(0) \m{x}_0 + \m{y}_2(0)\m{p}_0\tr
- \nabla C(\m{\bar{x}}(T)) \m{x}(T),
\end{equation}
where $\langle \cdot, \cdot \rangle$ denotes the $L^2$ inner product on
$[0, T]$, and the boundary conditions in (\ref{coupled}), (\ref{case2-x}),
and (\ref{case2-p}) are used to simplify the boundary terms.

We now combine (\ref{int}) and (\ref{boundary}), differentiate
the resulting identity with respect to $\m{p}_0$, and evaluate the
derivative at $\m{p}_0 = \m{\bar{p}}_0$.
Recall that $\m{y}$ is independent of $\m{p}_0$ since it corresponds to
(\ref{case2-x}) and (\ref{case2-p}) in the special case
where $(\m{x}, \m{p}) = (\m{\bar{x}}, \m{\bar{p}})$.
The only terms depending on $\m{p}_0$ are those involving $\m{x}$ and $\m{p}$,
the solution of (\ref{coupled}).
In particular, the partial derivatives of the three boundary terms
$\m{y}_1(0) \m{x}_0$, $\m{y}_2(0)\m{p}_0\tr$, and
$\nabla C(\m{\bar{x}}(T)) \m{x}(T)$ with respect to $\m{p}_0$
are zero, $\m{y}_2(0)$ and
$\nabla C(\m{\bar{x}}(T))\partial \m{x}(T)/\partial \m{p}_0$ respectively.
When we differentiate (\ref{int}) and (\ref{boundary}) with respect to
$\m{p}_0$ and evaluate at $\m{p}_0 = \m{\bar{p}}_0$,
every term cancels except for these three terms
(and one of these three terms is zero).
To see how these terms in the integrals cancel,
let us consider those terms with the
common factor $(\partial \m{x} / \partial \m{p}_0)(t)$.
This factor in the integral is multiplied by
\[
\dot{\m{y}}_1 (t) + \nabla_x \C{H}_j(\m{x}(t), \m{p}(t), \m{y}(t), t),
\]
which vanishes for $\m{x} = \m{\bar{x}}$ and $\m{p} = \m{\bar{p}}$ by
(\ref{case2-x}).
The remaining terms with the common factor
$(\partial \m{p} / \partial \m{p}_0)(t)$ vanish due (\ref{case2-p}).
After taking into account the three boundary terms in (\ref{boundary}),
we obtain
\begin{equation}\label{partial_p0}
\left. \frac{\partial C(\m{s}, \m{p}_0)}{\partial \m{p}_0}
\right|_{\m{p}_0 = \m{\bar{p}}_0} =
\nabla C(\m{\bar{x}}(T))\left. \frac{\partial \m{x}(T)}{\partial \m{p}_0}
\right|_{\m{p}_0 = \m{\bar{p}}_0} = \m{y}_2(0).
\end{equation}

In summary, the gradient of the objective with respect to
$\m{p}_0$ is available almost for free after evaluating
the derivative of the objective with respect to the switching points;
the gradient is simply $\m{y}_2(0)$.
As pointed out in Remark~\ref{nop},
$\m{y}_2 = \m{0}$ when the $\g{\phi}_j$ are independent of $\m{p}$.

\section{Free Terminal Time}
\label{free}
So far, the terminal time $T$ has been fixed.
Let us now suppose that the terminal time is free, and we are minimizing
over both the terminal time $T$ and over the control $\m{u}$.
It is assumed that the control constraint set $\C{U}$ is independent of $t$,
and we make the change of variable $t = \tau T$,
where $0 \le \tau \le 1$.
After the change of variables,  both the state and the
control are functions of $\tau$ rather than $t$.
The reformulated optimization problem is
\begin{equation}\label{tau}
\min C(\m{x}(1)) \quad \mbox{subject to} \quad \m{\dot{x}}(\tau) =
T \m{f}(\m{x}(\tau), \m{u}(\tau)), \quad \m{x}(0) = \m{x}_0, \quad
\m{u}(\tau) \in \C{U},
\end{equation}
where $\m{x} : [0, 1] \rightarrow \mathbb{R}^n$ is absolutely continuous and
$\m{u} : [0, 1] \rightarrow \mathbb{R}^m$ is essentially bounded.
In the reformulated problem, the free terminal time $T$ appears as a parameter
in the system dynamics.

For fixed $T$, the optimization problem over the control has the same
structure as that of the problem analyzed in Sections~\ref{dscase1} and
\ref{dscase2}.
Hence, the previously derived formula for the derivative of the objective
with respect to a switching point remains applicable.
If a gradient-based algorithm will be used to solve (\ref{tau}),
then we also need a formula for the derivative of the
objective with respect to $T$ when the switch points for the control
are fixed.
Since the switch points for the control are fixed throughout this section,
the objective value in (\ref{tau}) only depends on the choice of
the parameter $T$ in the dynamics.
Assuming that for some given $T$ there exists a solution $\m{x}$
to the dynamics in (\ref{tau}),
we let $C(T) := C(\m{x}(1))$ denote the objective value.
By the chain rule,
\begin{equation}\label{dT}
\frac{d C(T) }{dT} = \nabla C[\m{x}(1)] \frac{ d \m{x}(1)} {dT}.
\end{equation}
Similar to the approach in Section~\ref{dscase2}, our goal is to obtain
an expression for the right side of (\ref{dT}) that avoids the computation
of the derivative of the state with respect to $T$.
Let us first consider Case~1 where the control
has the form $\m{u}(\tau) = \g{\phi}_j(\m{x}(\tau), \tau)$
for all $\tau \in (s_j, s_{j+1})$, $0 \le j \le k$.
\smallskip

\begin{theorem} \label{Ttheorem}
Suppose that for $T = \bar{T}$,
$\m{x} = \m{\bar{x}}$ is an absolutely continuous solution of
\begin{equation}\label{IVPT}
\m{\dot{x}}(\tau) = T \m{F}(\m{x}(\tau), \tau), \quad \m{x}(0) = \m{x}_0, \quad
0 \le \tau \le 1,
\end{equation}
where $\m{F}$ is defined in $(\ref{IVP})$.
We assume that for some $\rho > 0$,
$\m{f}_j(\g{\chi}, \tau)$, $0 \le j \le k$, is continuous with respect to
$\g{\chi}$ and $\tau$ and Lipschitz continuous
with respect to $\g{\chi}$ on the tube
\[
\{ (\g{\chi}, \tau) :
\tau \in [s_j, s_{j+1}] \mbox{ and }
\g{\chi} \in \C{B}_\rho(\m{x}(\tau)) \}, \quad 0 \le j \le k.
\]
Then we have
\begin{equation}\label{dCdT}
\left. \frac{dC (T)}{dT}\right|_{T = \bar{T}} =
\int_0^1 H(\m{\bar{x}}(\tau), \m{p}(\tau), \tau) \; d\tau ,
\end{equation}
where $H(\m{x}, \m{p}, \tau) = H_j (\m{x}, \m{p}, \tau)$ when
$s_j \le \tau \le s_{j+1}$,
and the row vector $\m{p}: [0, 1] \rightarrow \mathbb{R}^n$
is the solution to the linear differential equation
\begin{equation}\label{p-caseT}
\m{\dot{p}}(\tau) = - \bar{T} \m{p}(\tau) \nabla_x \m{F}(\m{\bar{x}}(\tau), \tau),
\quad \tau \in [0, 1],
\quad
\m{p}(1) = \nabla C[\m{\bar{x}}(1)] .
\end{equation}
\end{theorem}
\smallskip

\begin{proof}
If $\m{p}$ denotes the costate given by (\ref{p-caseT}),
and $\m{x}$ for $T$ near $\bar{T}$ denotes the solution of (\ref{IVPT}),
then we have the identity
\begin{eqnarray}
0 &=& \int_0^1
\m{p}(\tau) [T \m{F}(\m{x}(\tau), \tau) - \m{\dot{x}}(\tau) ] \; d\tau
\nonumber \\
&=& \m{p}(0)\m{x}_0 - \m{p}(1)
\m{x} (1) + \int_0^1
\Big[ T\m{p}(\tau)\m{F}(\m{x}(\tau), \tau)
+ \m{\dot{p}}(\tau) \m{x}(\tau) \Big] \; d\tau , \label{T}
\end{eqnarray}
where the second equation comes from an integration by parts.
Let us differentiate with respect to $T$.
Since $\m{p}$ in (\ref{p-caseT}) does not depend on $T$, the
derivative of $\m{p}(0)\m{x}_0$ with respect to $T$ is zero.
Hence, the derivative of (\ref{T}) with respect to $T$,
evaluated at $T = \bar{T}$, yields
\[
\m{p}(1) \left. \frac{d \m{x}(1)}{dT}\right|_{T = \bar{T}} =
\int_0^1 \Big\{ \m{p}(\tau)\m{F}(\m{\bar{x}}(\tau), \tau)
+ \big[ \bar{T}\nabla_x \m{F}(\m{\bar{x}}(\tau), \tau)
+ \m{\dot{p}}(\tau) \big]
\left. \frac{d \m{x}(\tau)}{dT}\right|_{T = \bar{T}} \Big\}
d\tau .
\]
Substituting for $\m{p}$ from (\ref{p-caseT}), the factor
multiplying $d\m{x}(\tau)/dT$ is zero.
It follows from (\ref{dT}) that
\[
\left. \frac{dC (T)}{dT}\right|_{T = \bar{T}} =
\nabla C[\m{\bar{x}}(1)] \left. \frac{ d \m{x}(1)} {dT} \right|_{T=\bar{T}} =
\int_0^1 H(\m{\bar{x}}(\tau), \m{p}(\tau), \tau) \; d\tau ,
\]
which completes the proof.
\end{proof}
\smallskip

\begin{remark}Note that the Hamiltonian in $(\ref{dCdT})$ does not contain the
terminal time, while the Hamiltonian associated with $(\ref{IVPT})$
and the objective derivative with respect to a switch point does
contain the terminal time.
It follows from $(\ref{dCdT})$ that the integral of the Hamiltonian vanishes
along an optimal solution to the control problem $(\ref{CP})$.
Hence, due to the constancy of the Hamiltonian along an optimal solution,
$(\ref{dCdT})$ implies that the Hamiltonian vanishes along an optimal solution,
a classic first-order optimality condition for free terminal time
control problems.
\end{remark}
\smallskip

Now consider Case~2 where the control
has the form $\m{u}(\tau) = \g{\phi}_j(\m{x}(\tau), \m{p}(\tau), \tau)$
for all $\tau \in (s_j, s_{j+1})$, $0 \le j \le k$.
The generalized state $(\m{x}, \m{p})$ satisfies the coupled system
\begin{equation} \label{coupledT}
\m{\dot{\m{x}}}(t) = T\m{F}(\m{x}(t), \m{p}(t), t), \quad
\m{\dot{\m{p}}}(t) = -T\m{p}(t)\m{F}_x(\m{x}(t), \m{p}(t), t),
\end{equation}
where $(\m{x}(0), \m{p}(0)) = (\m{x}_0, \m{p}_0)$ and the terminal time
$T$ is a parameter in the equations.
For fixed $T$ and a fixed choice of the switching times $\m{s}$, the
derivative of the objective with respect to $\m{p}_0$ is given by
the formula (\ref{partial_p0}).
Now for a fixed choice of both $\m{s}$ and $\m{p}_0$, say
$\m{p}_0 = \m{\bar{p}}_0$,
our goal is to evaluate the objective derivative with respect to $T$
evaluated at some given terminal time $T = \bar{T}$,
assuming a solution
$(\m{\bar{x}}, \m{\bar{p}})$ to (\ref{coupledT}) exists for
$T = \bar{T}$ and $\m{p}_0 = \m{\bar{p}}_0$.

The analogue of (\ref{T}) is a slightly modified version of
(\ref{int}) where the integration limit $T$ is replaced by 1, while
$T$ appears as a parameter next to the dynamics:
\begin{equation}\label{GT}
0 = \int_0^1 \m{y}_1(t) [T\m{F}(\m{x}(t), \m{p}(t), t) - \m{\dot{x}}(t)]
- [T\m{p}(t)\m{F}_x(\m{x}(t), \m{p}(t), t) + \m{\dot{p}}(t)] \m{y}_2\tr(t)
\; dt,
\end{equation}
where $(\m{x}, \m{p})$ is the solution to (\ref{coupledT}) corresponding
to a general $T$, but with the initial condition $\m{p}_0 = \m{\bar{p}}_0$.
The generalized costate $\m{y}$ in (\ref{GT}) is the solution to
\begin{eqnarray*}
\dot{\m{y}}_1 (t) &=& -\bar{T}\nabla_x \C{H}_j(\m{\bar{x}}(t), \m{\bar{p}}(t), \m{y}(t), t),
\quad
\m{y}_1(T) = \nabla C(\m{\bar{x}}(T)) \\
\dot{\m{y}}_2 (t) &=& -\bar{T}\nabla_p \C{H}_j(\m{\bar{x}}(t), \m{\bar{p}}(t), \m{y}(t), t),
\quad \m{y}_2(T) = \m{0}
\end{eqnarray*}
on $(s_{j}, s_{j+1})$ for $j = k,~ k-1,~ \ldots, ~0$.
Proceeding as in the proof of Theorem~\ref{Ttheorem},
we integrate by parts in (\ref{GT}), differentiate with respect to $T$,
and evaluate at $T = \bar{T}$ to obtain the relation
\[
\left. \frac{dC (T)}{dT}\right|_{T = \bar{T}} =
\nabla C[\m{\bar{x}}(1)] \left. \frac{ d \m{x}(1)} {dT} \right|_{T=\bar{T}} =
\int_0^1 \C{H}(\m{\bar{x}}(\tau), \m{\bar{p}}(\tau), \m{y}(\tau), \tau) \;
d\tau ,
\]
where $\C{H}(\m{x}, \m{p}, \m{y}, \tau) =$
$\C{H}_j (\m{x}, \m{p}, \m{y}, \tau)$ when
$s_j \le \tau \le s_{j+1}$, $0 \le j \le k$.
\section{Starting Guess}
\label{starting}
This section discusses how to generate a starting guess for the
Switch Point Algorithm.
Detailed illustrations of these techniques are given in \cite{Atkins20}.
If the optimal control is bang-bang without singular intervals, then the
optimization problem could be discretized by Euler's method, and the
location of the switching point can often be estimated with a few iterations
of a gradient or a conjugate gradient method.
On the other hand, when a singular interval is present, wild oscillations
in the control can occur and the problem becomes more difficult.
An effective way to approximate the optimal control in the singular setting
is to incorporate total variation (TV) regularization in the objective.
TV regularization has been very effective in image restoration since it
preserves sharp edges; for singular optimal control problems, it helps to
remove the wild oscillations in the control, and better exposes the
switch points.

We consider the Euler discretization of (\ref{CP}) with
$\rho$-amplified TV regularization:
\begin{eqnarray}
&\min C(\m{x}_N) + \rho \sum_{i=1}^m \sum_{j = 1}^{N-1} |u_{ij} - u_{i,j-1}|&
\label{D} \\
&\mbox{subject to} \;\; \m{x}_{j+1} = \m{x}_j +
h \m{f}(\m{x}_j, \m{u}_j), ~ \m{u}_j \in \C{U}(t_j),& \nonumber
\end{eqnarray}
where $0 \le j \le N-1$, $h = T/N$, $t_j = jh$, and
$N$ is the number of mesh intervals.
The parameter $\rho$ controls the strength of the TV regularization term,
and as $\rho$ increases, the oscillations in $u$ should decrease.
The nonsmooth problem (\ref{D}) is equivalent to the smooth
optimization problem
\begin{eqnarray}
&\min C(\m{x}_N) + \rho \sum_{i=1}^m \sum_{j = 1}^{N-2}
 v_{ij} + w_{ij} .&
\label{SD} \\
&\mbox{s.t.} \; \m{x}_{j+1} = \m{x}_j +
h \m{f}(\m{x}_j, \m{u}_j),
~ \m{u}_j \in \C{U}(t_j),
~ \m{u}_{l+1} - \m{u}_l = \m{v}_l - \m{w}_l,
~ \m{v}_l \ge \m{0}, ~ \m{w}_l \ge \m{0},
& \nonumber
\end{eqnarray}
where $0 \le j \le N-1$ and $0 \le l \le N-2$.
The equivalence between (\ref{D}) and (\ref{SD}) is due to the well-known
property in optimization that
\[
|u| = \min \{ v + w : u = v - w, ~v \ge 0, ~w \ge 0 \}.
\]
The smooth TV-regularized problem (\ref{SD}) can be solved quickly by
the Polyhedral Active Set Algorithm \cite{HagerZhang16} due to the
sparsity of the linear constraints.

Figure~\ref{catmixfig} shows how the optimal control of (\ref{D})
for the catalyst mixing problem of the next section depends on $\rho$.
When $\rho = 0$, the control oscillates wildly, when $\rho = 10^{-5}$
many of the oscillations are gone, and when $\rho = 10^{-3}$ the computed
solution provides a good fit to the exact solution, and the switching points
for the discrete problem are roughly within the mesh spacing ($N = 100$)
of the exact switching points.
In some problems with highly oscillatory solutions,
convergence of TV regularized optimal values is established in
\cite{Caponigro18}.

When solving (\ref{D}), we also obtain an estimate for the initial costate
$\m{\bar{p}}(0)$ associated with a solution of (\ref{CP}).
In particular, the KKT conditions at a solution of (\ref{D}) yield
the following equation for the multiplier $\m{p}_j$ associated with
the constraint $\m{x}_j + h\m{f}(\m{x}_j, \m{u}_j) - \m{x}_{j+1} = \m{0}$:
\[
\m{p}_{j-1} = \m{p}_j (\m{I} + h\nabla_x \m{f}(\m{x}_j , \m{u}_j)), \quad
1 \le j \le N-1, \quad \m{p}_{N-1} = \nabla C(\m{x}_N) ,
\]
where $\m{p}_0$ is an approximation to $\m{\bar{p}}(0)$ and $\m{I}$ is
the $n$ by $n$ identity matrix.
\begin{figure}
\begin{nscenter}
\includegraphics[scale=.27]{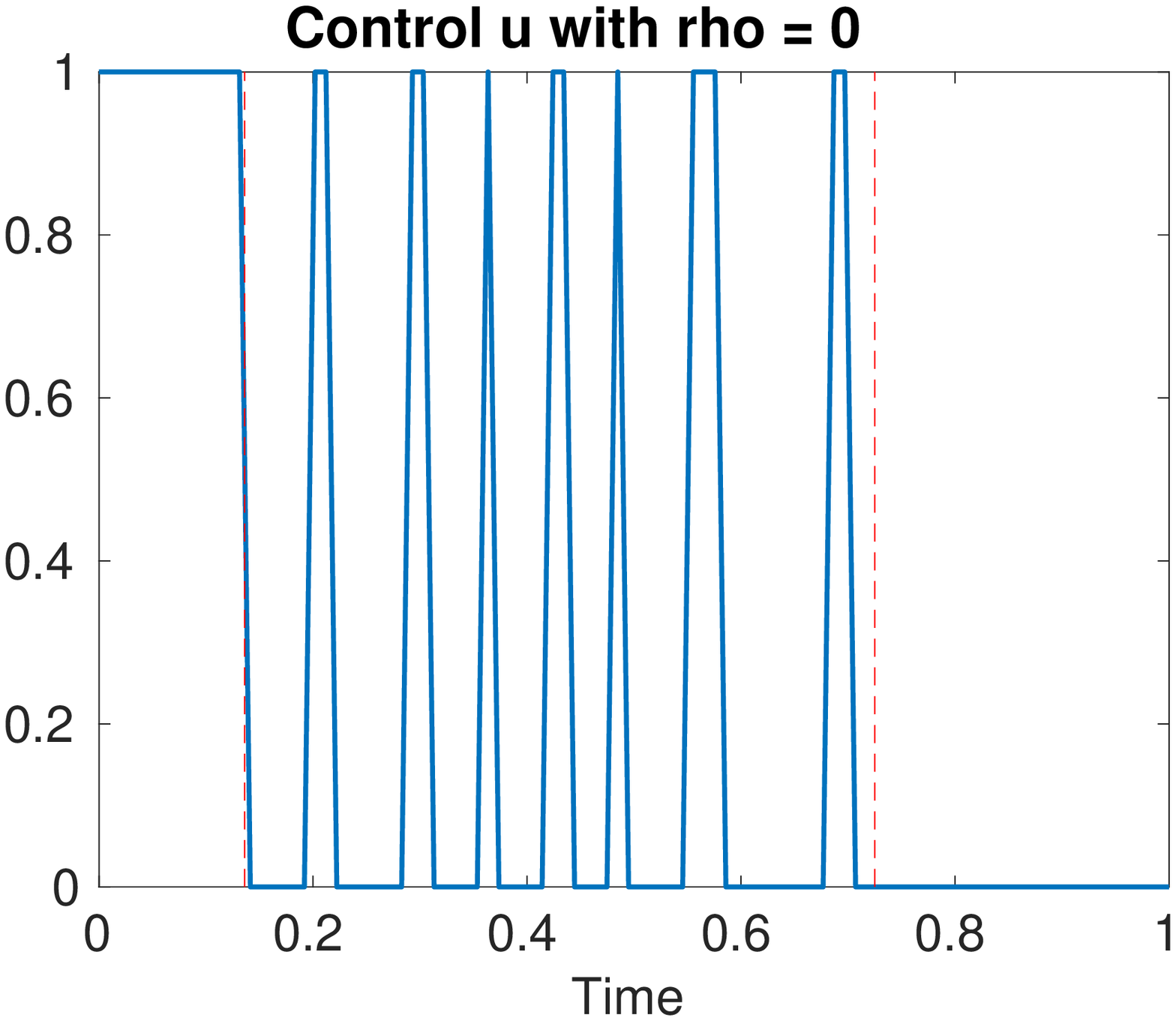}
\includegraphics[scale=.27]{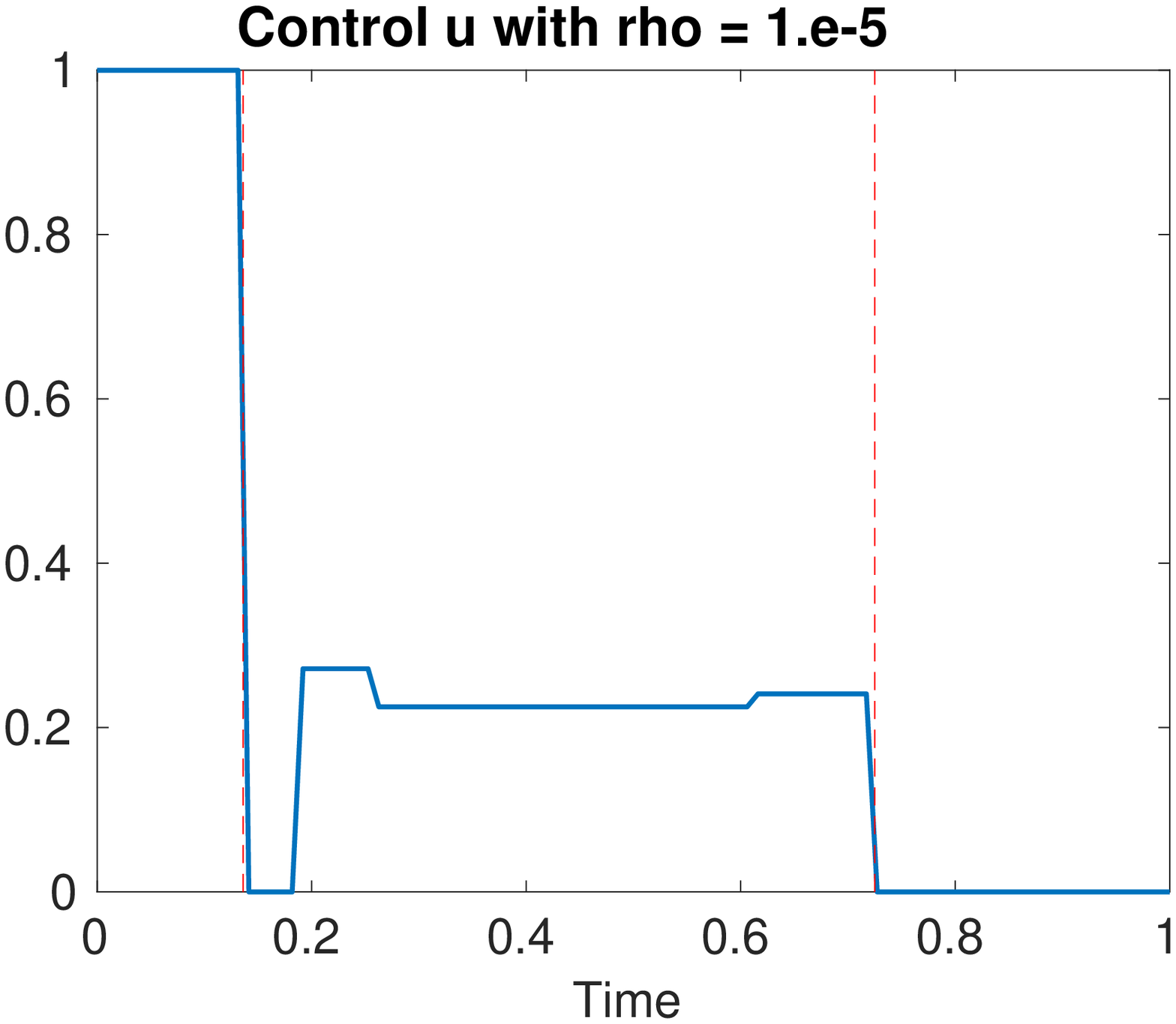}
\includegraphics[scale=.27]{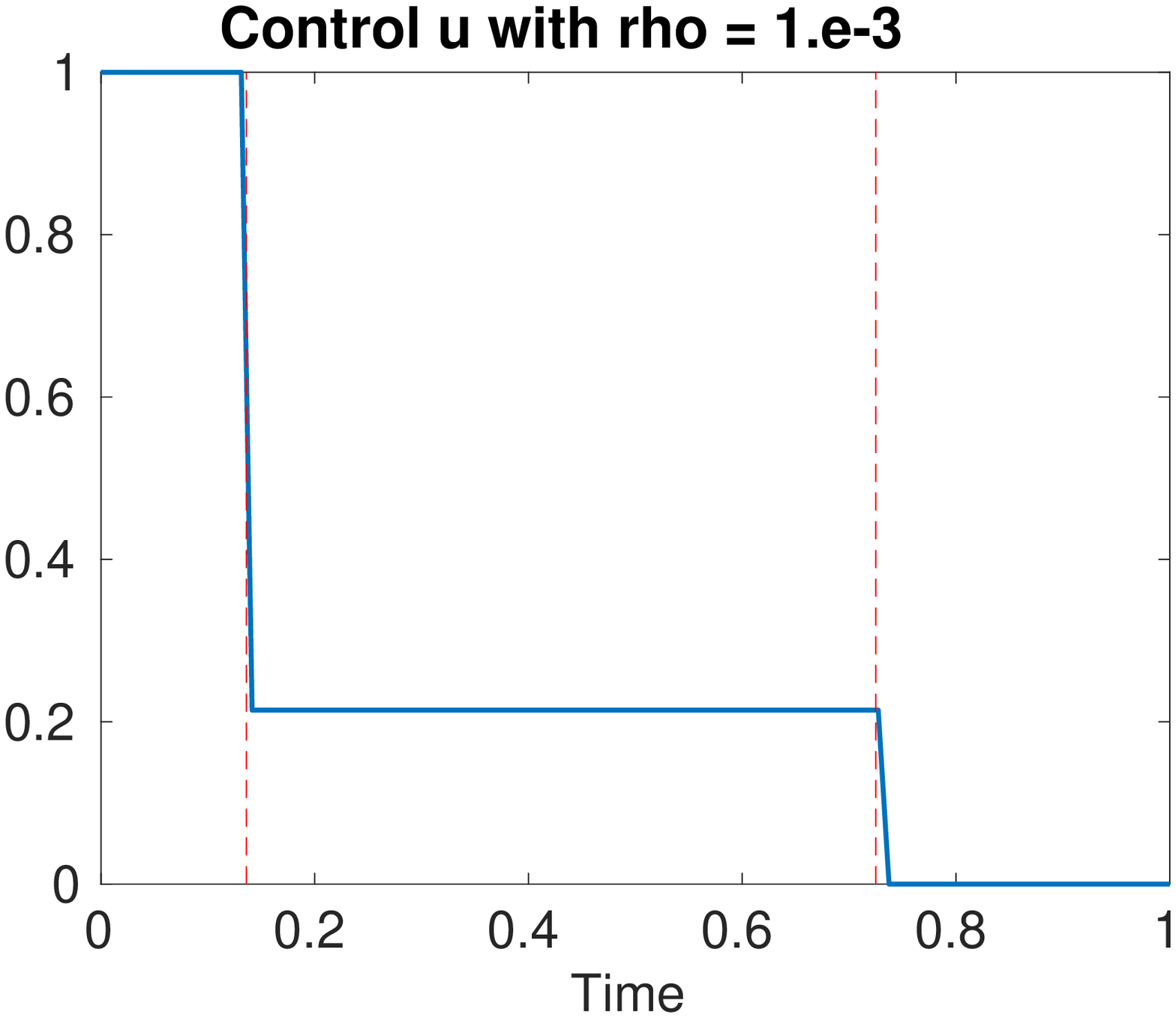}
\caption{Numerical solutions to $(\ref{D})$ for the catalyst mixing problem
and three different
choices for the regularization: $\rho = 0, 10^{-5}, 10^{-3}$.
Exact switching points appear as dashed lines.
\label{catmixfig}}
\end{nscenter}
\vspace*{.1in}
\end{figure}

\section{Numerical Studies}
\label{examples}
We consider four classic singular control problems from the literature
to examine the performance of the Switch Point Algorithm relative to
that of previously reported algorithms.
The test problems are the following:
\begin{enumerate}
\item
The catalyst mixing problem, first proposed by Gunn and Thomas
\cite{Gunn1965MassTA}, and later solved by Jackson \cite{Jackson68}, with
additional analytic formulas given by Li \cite{Li2015CommentsO}.
\item
A problem posed by Jacobson, Gershwin, and Lele  \cite{Jacobson70}.
\item
A problem posed by Bressan \cite{Bressan11}.
\item
A problem posed by Goddard \cite{Goddard1919}.
\end{enumerate}
All the test problems have known solutions.
The Switch Point Algorithm was implemented in MATLAB; the gradients
of the objective were evaluated using the formulas given in the paper.
The differential equations were integrated using MATLAB's ODE45 code,
which implements the Dormand-Prince \cite{DormandPrince80} embedded
explicit Runge-Kutta (4,5) scheme
(both 4th and 5th order accuracy with error estimation).
The optimization was performed using the
Polyhedral Active Set Algorithm (PASA) \cite{HagerZhang16}
(available from Hager's web page).
The experiments were performed on a
Dell T7610 workstation with 3.40~GHz processors.
We did not implement other algorithms;
we simply compare to previously reported results in the literature.

\subsection{Catalyst mixing problem}
The optimal control problem is as follows:
\begin{equation}\label{catdy}
    \begin{split}
        \text{min}~ & a(T) + b(T) - 1\\
        \text{s.t.}~ & \dot{a}(t) = -u(t)(k_1 a(t) - k_2 b(t) ) \\
        & \dot{b}(t) = u(t) (k_1 a(t) - k_2 b(t))- ( 1 - u(t))k_3 b(t)\\
        & a(0) = 1,~b(0) = 0, ~ 0 \le u(t) \le 1
    \end{split}
\end{equation}
Here $a$ and $b$ are the mole fractions of substances $A$ and $B$ which
are catalyzed by fraction $u$ to produce $C$ in a reactor of length $T$.
The parameters $k_i$, $i = 1,2,3$, are given constants, and the objective
corresponds to the maximization of $C$.
Symbolically, the relation is denoted $A \rightleftharpoons  B \rightarrow C$.
As seen in Figure~\ref{catmixfig}, the pattern of the optimal control
is bang, singular, and off.

The Hamiltonian for this problem is
\[
H(a, b, u, p_1, p_2) =
\big[ (p_2-p_1)(k_1a - k_2b) + p_2k_3b] u - p_2k_3b .
\]
The switching function, which corresponds to the coefficient of $u$, is
given by
\begin{equation}\label{hamfcn}
\C{S} (t) = 
(p_2(t)-p_1(t))(k_1a(t) - k_2b(t)) + k_3 p_2(t)b(t) ,
\end{equation}
where $\m{p}$ is the solution of the system
\[
\begin{array}{llll}
\dot{p}_1(t) &=& - (p_2(t)  - p_1(t))k_1 u(t), & p_1 (T) = 1,\\[.05in]
\dot{p}_2(t) &=& (p_2(t)  - p_1(t)) k_2 u(t) + k_3(1-u(t))p_2(t) ,
& p_2 (T) = 1. \end{array}
\]

When $\C{S}(t) < 0$, $u (t) = 1$; when $\C{S}(t) > 0$, $u(t) = 0$; and
when $\C{S}(t) = 0$, $u$ is singular.
In the singular region, both $\C{S}$ and its derivatives vanish.
Differentiating the switch function, we have
\begin{eqnarray*}
\dot{\C{S}}(t) &=& k_3 \big[ k_1 a(t) p_2(t) - k_2 b(t) p_1(t) \big] ,\\
\Ddot{\C{S}}(t)&=& k_3 \Big\{ u(t) p_1(t) [k_2 b(t) (k_2 - k_3 - k_1)
- 2k_1 k_2 a(t)] + \\
&& \quad \;\; u(t) p_2(t) [k_1 a(t) (k_2 - k_3 - k_1)+ 2k_1 k_2 b(t)] + \\
&& \quad \;\; k_3 \big[ k_1 p_2 (t) a(t) + k_2 p_1(t) b(t)\big] \Big\} .
\end{eqnarray*}
Although the control is missing from $\dot{\C{S}}$, it appears in
$\Ddot{\C{S}}$.
Hence, we use the equation $\Ddot{\C{S}}(t) = 0$ to solve for the
control in the singular region.
In the following equation, the ``$(t)$'' arguments for the state and costate
are omitted:
\begin{equation}\label{u1}
u_{\rm {sing}} =
\frac{-k_3 (k_1 ap_2 + k_2 b p_1)}
{p_1 [k_2 b(k_2 - k_3 - k_1) - 2k_1 k_2 a]
+ p_2 [k_1 a (k_2 - k_3 - k_1) + 2k_1 k_2 b]}
\end{equation}

With further analysis, it can be shown that the singular control is constant.
Jackson \cite{Jackson68} derives the following expression
for the singular control, where the numeric value given below corresponds
to the parameter values $k_1 = k_3 = 1$ and $k_2 = 10$ which are used throughout
the literature:
\begin{equation}\label{u2}
u_{\mbox{\footnotesize {sing}}} = \frac{\alpha (1+\alpha)}{\beta + (1+\alpha)^2}
\approx 0.227142082708498,
\end{equation}
where $\alpha = \sqrt{k_3/k_2}$ and $\beta = k_1/k_2$.

There are two switch points for the catalyst mixing problem.
Analytic formulas for the switching times, presented by Jackson, are
\begin{eqnarray*}
s_1 &=&  \left( \frac{1}{k_2(1 + \beta)} \right)
~ \log \left( \frac{1+ \alpha + \beta}{\alpha} \right) \approx
0.136299034594555, \\
s_2 &=& T - k_3^{-1} \log (1 + \alpha) \approx T - 0.274769892408345.
\end{eqnarray*}
The optimal control is
\[
u(t) = \left\{ \begin{array}{ll}
1 & \mbox{if } 0 \le t < s_1, \\
u_{\mbox{\footnotesize{sing}}} & \mbox{if } s_1 \le t < s_2, \\
0 & \mbox{if } s_2 \le t \le T. \end{array} \right.
\]

An analytic formula is given for the optimal objective value in 
\cite{Li2015CommentsO};
the numeric values for the optimal objectives are
\[
\begin{array}{ll}
-0.048055685860877 & (T = 1), \\
-0.191814356325161 & (T = 4), \\
-0.477712020050041 & (T = 12).
\end{array}
\]

We solve the test problem using both the Case~1 representation
in (\ref{u2}), where the exact form of the singular control is exploited,
and the Case~2 representation for the
singular control given in (\ref{u1}) where the algorithm computes the
control in the singular region.
The versions of the Switch Point Algorithm corresponding to these
two representation of the singular control are denoted
SPA1 (Case~1) and SPA2 (Case~2).

In Table \ref{data1}
we compare the performance of SPA1 and SPA2 to results from the literature
for the problem (\ref{catdy}) with reactor lengths $T = 1$, 4, and 12.
The accuracy tolerances used for PASA and ODE45 were $10^{-8}$.
The starting guesses for the switching times and the initial costate
were accurate
to roughly one significant digit.
For example, with $T = 1$ the starting guesses were
$s_1 = 0.1$, $s_2 = 0.7$, $p_1(0) = 0.9$, and $p_2(0) = 0.8$.
As seen in Table~\ref{data1}, the accuracy was improved from the initial one
significant digit to between 9 and 11 significant digits when using
$10^{-8}$ tolerances for both PASA and ODE45.  
\begin{table}[ptbh]
\begin{center}%
\begin{tabular} [c]
{|c|c|r|c|c|c|}\hline
Method & $T$ &CPU (s)& $C$ Error& $s_1$ Error & $s_2$ Error \\ \hline
SPA1 & 1 & $ 0.10$& 1.6$\times 10^{-10}$ & 3.1$\times 10^{-09}$&1.2$\times 10^{-11}$\\
SPA2 & 1 & $ 0.30$& 9.6$\times 10^{-12}$ & 1.9$\times 10^{-10}$&6.2$\times 10^{-11}$\\
\cite{Banga1998Stochastic}& 1 & $40-100$& 5.7$\times 10^{-06}$ & $-$ & $-$\\
\cite{Bell2000}&1 & $9.74$ & 2.4$\times 10^{-05}$& $-$ & $-$\\
\cite{Dadebo1995Dynamic} & 1 & $-\;\;$ & 5.7$\times 10^{-06}$&$-$&$-$ \\
\cite{Foroozandeh2018}& 1 & $-\;\;$ & 6.3$\times 10^{-09}$ & 3.0$\times 10^{-09}$ & 1.1$\times 10^{-12}$  \\
\cite{Irizarry2005} & 1 & $-\;\;$ & 5.7$\times 10^{-06}$ & $-$ & $-$ \\
\cite{Li2015CommentsO}& 1 & $-\;\;$ & 5.9$\times 10^{-09}$ & 1.5$\times 10^{-08}$ & 6.8$\times 10^{-08}$\\
\cite{Liu2013SolutionOC}& 1 & $17.90$& 1.4$\times 10^{-08}$ &$-$&$-$  \\
IDE \cite{Lobato2011} & 1 & $-\;\;$ & 2.3$\times 10^{-05}$& 8.3$\times 10^{-03}$& 7.8$\times 10^{-03}$\\
\cite{TanartkitLorenz}& 1 & 90.00 & 1.4$\times 10^{-05}$ &$-$&$-$ \\
\cite{Vassiliadis1993} & 1 & 38.10 & 1.4$\times 10^{-08}$ & 2.1$\times 10^{-05}$&4.9$\times 10^{-04}$  \\
SPA1 & 4 & 0.12& 1.1$\times 10^{-10}$ & 4.5$\times 10^{-09}$&1.5$\times 10^{-09}$\\
SPA2 & 4 & 0.68& 1.4$\times 10^{-10}$ & 2.4$\times 10^{-10}$& 1.5$\times 10^{-11}$\\
\cite{Biegler20} & 4 & 0.90 & $-$ & 4.6$\times 10^{-09}$
& 7.6$\times 10^{-09}$ \\
\cite{Biegler19b}& 4 & 0.33 &$-$& 3.3$\times 10^{-03}$& 4.8$\times 10^{-03}$ \\
\cite{Biegler16}& 4 & 1.35 & $-$ & 8.4$\times 10^{-08} $& 3.2$\times 10^{-07}$\\
\cite{Biegler19a} & 4 & 0.90 & $-$ & 5.0$\times 10^{-07}$& 8.3$\times 10^{-06} $\\
SPA1 & 12& 0.19& 1.7$\times 10^{-10}$ & 3.7$\times 10^{-10} $& 4.4$\times 10^{-08}$\\
SPA2 & 12& 0.98& 2.0$\times 10^{-11}$ & 1.6$\times 10^{-09} $& 3.6$\times 10^{-14}$\\
\cite{Dadebo1995Dynamic} & 12 &595.00& 7.7$\times 10^{-04}$ &$-$&$-$ \\
\cite{Li2015CommentsO} & 12 & $-\;\;$ & 5.0$\times 10^{-11}$ & 1.1$\times 10^{-07} $& 4.0$\times 10^{-07} $ \\
\cite{Liu2013SolutionOC} & 12 & 17.79& 7.7$\times 10^{-04} $ &$-$&$-$  \\
\cite{Pham2008DynamicOO} & 12 & $-\;\;$ & 2.7$\times 10^{-04}$ &$-$&$-$ \\
\cite{Rajesh2001DynamicOO} & 12 &0.24& 1.6$\times 10^{-03}$ &$-$&$-$ \\
\hline
\end{tabular}
\caption{Performance and absolute errors for the catalyst mixing problem.
\label{data1}}%
\end{center}
\end{table}

Note that many of the algorithms in Table~\ref{data1} exploit the known form
of the singular control; the computing times for SPA1 where the known form
of the singular control is exploited are significantly smaller than the
time for SPA2 where the singular control is computed.
It is difficult to compare the computing times in Table~\ref{data1} since the
computers used to solve the test problem vary widely in speed.
Moreover, it should be possible to significantly lower the computing time for
the SPA by developing an implementation in a compiled
language instead of MATLAB.
And with an ODE integrator tailored to the structure of the control
problem, the computing time could be reduced further.
Observe that the accuracy of the solution computed by
SPA was relatively high when using a modest
$10^{-8}$ accuracy tolerance for both the optimizer and the ODE integrator.

\subsection{Jacobson's Problem \cite{Jacobson70}} 
The test problem is given by
\[
\begin{split}
  \min~ & \frac{1}{2}\int_0^5 x_1^2 (t)+x_2^2 (t)\; dt\\
  \mbox{s.t.} ~&\dot{x}_1(t)=x_2(t), \quad \dot{x}_2(t)=u(t), \\
&x_1 (0) = 0,~ x_2(0) = 1, \quad -1 \leq u(t) \leq 1.
\end{split}
\]
This problem as well as the next can be reformulated in form
(\ref{CP}) by adding a new state variable whose dynamics is the
integrand of the objective.
After this reformulation, one finds that the costate
associated with the new variable is $p_0 := 1$, so the Hamiltonian
simplifies to
\[
H(\m{x}, u, \m{p}) =
\frac{1}{2} (x_1^2 + x_2^2) + p_1 x_2 + p_2 u .
\]
The switching function is $\C{S}(t) = p_2(t)$, where $\m{p}$ is the
solution of the system
\[
\begin{array}{ll}
\dot{p}_1(t) = -x_1(t),          & p_1(5) = 0, \\
\dot{p}_2(t) = -p_1(t) - x_2(t), & p_2(5) = 0. \\
\end{array}
\]
The first two derivatives of the switching function are
$\dot{\C{S}}(t) = -p_1(t) - x_2 (t)$ and $\ddot{\C{S}}(t) = x_1(t) - u(t)$.
In the singular region, $\ddot{\C{S}} = 0$, which implies that
$u(t) = x_1 (t)$.
The optimal control has one switching point whose first few digits are
\[
s_1 \approx 1.41376408763006415924,
\]
which is the root of the equation
\[
1-s^2/2 = e^{2s-10} (-1 + 2s - s^2/2) .
\]

The optimal control is
\[
u(t) = 
\begin{cases} 
-1 &  \text{if} ~ 0\leq t< s_1, \\
x_1(t) &  \text{if} ~ s_1 < t \le 5,
\end{cases}
\]
where $[s_1, 5] $ is the singular interval. 

Since this problem has only one switch point, we will compute it by
using the secant method to find a zero of the objective derivative
with respect to $s_1$.
The left panel of Figure~\ref{DC} plots the derivative of the objective function
with respect to the switch point.
Notice that the derivative vanishes twice, once on the interval
$[1.41, 1.42]$ (corresponding to the optimal control) and once on the
interval $[1.45,1.46]$ corresponding to a local maximum.
Thus to compute the correct switching point using the secant method,
we should start to the left of the local maximum in Figure~\ref{DC}.
The results in Table~\ref{data2} correspond to the starting iterates
1.42 and 1.41, which bracket the switch point of the optimal control.
In 5 iterations of the secant method, we obtain roughly 11 digit accuracy.
The solution time of SPA is basically the time of MATLAB's ODE45 integrator.
In Figure~\ref{DC} it appears that by choosing the switch point to the
right of the local maximum, it may be possible to achieve a smaller value
for the cost function.
However, for the switch point corresponding to the local
maximum in Figure~\ref{DC}, the singular control $u(t) = x_1(t)$ is
$-1$ at $t = 5$, and as the switching point moves further to the right,
the singular control becomes infeasible with $u(5) < -1$.
\begin{table}[h]
\begin{center}%
\begin{tabular} [c] {|c|r|l|}\hline
Method & CPU (s)& $s_1$ Error \\ \hline
SPA1 & 0.15 & 5.0$\times 10^{-11}$\\
\cite{Biegler20} & 2.34 & 6.8$\times 10^{-08}$ \\
\cite{Biegler19b}& 0.39 & 3.8$\times 10^{-03}$ \\
\cite{Biegler16}& 132.03 & 7.4$\times 10^{-06}$\\
\cite{Biegler19a} & 0.33 & 3.2$\times 10^{-07}$\\
\hline
\end{tabular}
\end{center}
\caption{Performance and absolute errors for Jacobson's problem
\label{data2}}
\end{table}
 
\begin{figure}
\begin{nscenter}
\includegraphics[scale=.31]{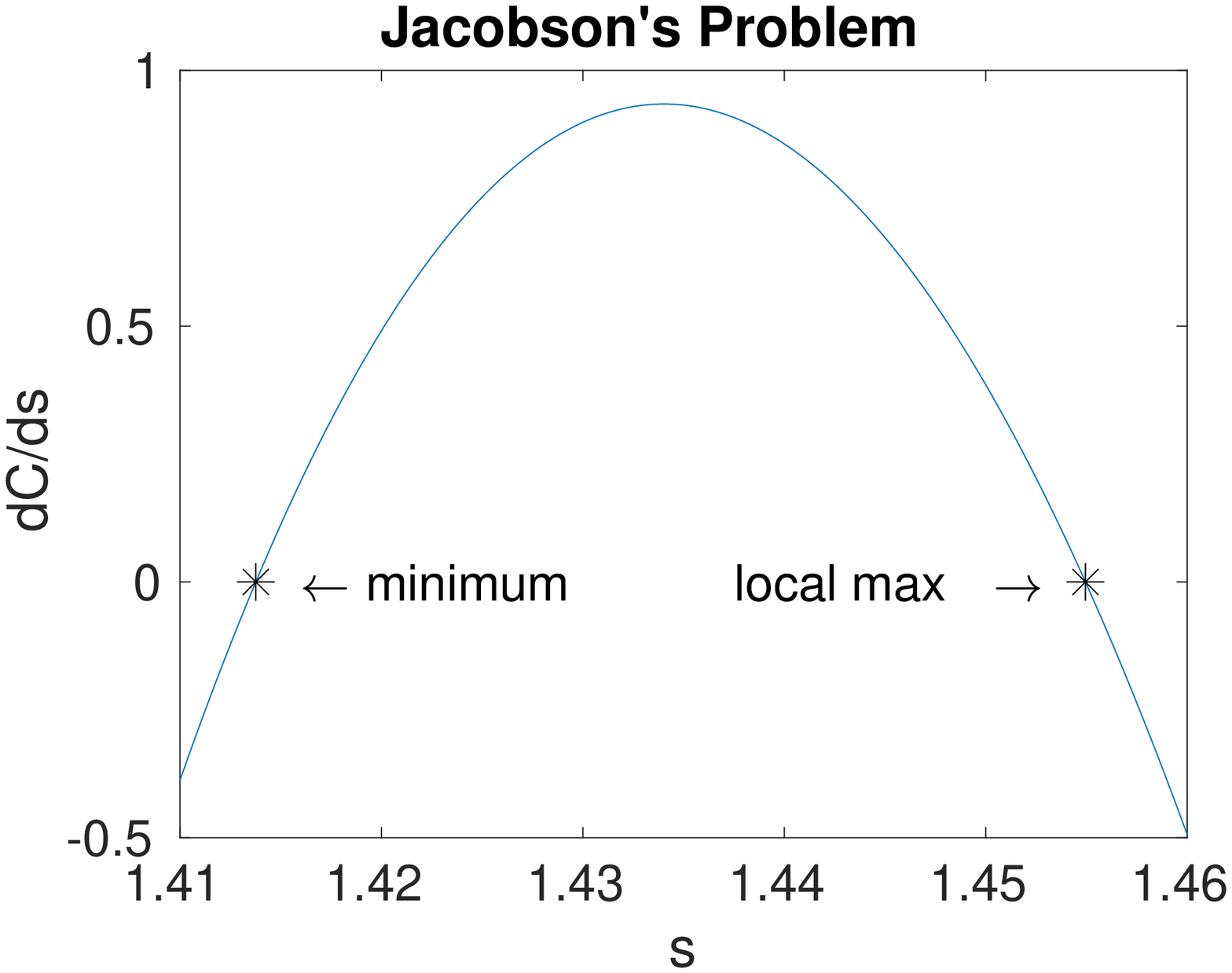}
\includegraphics[scale=.31]{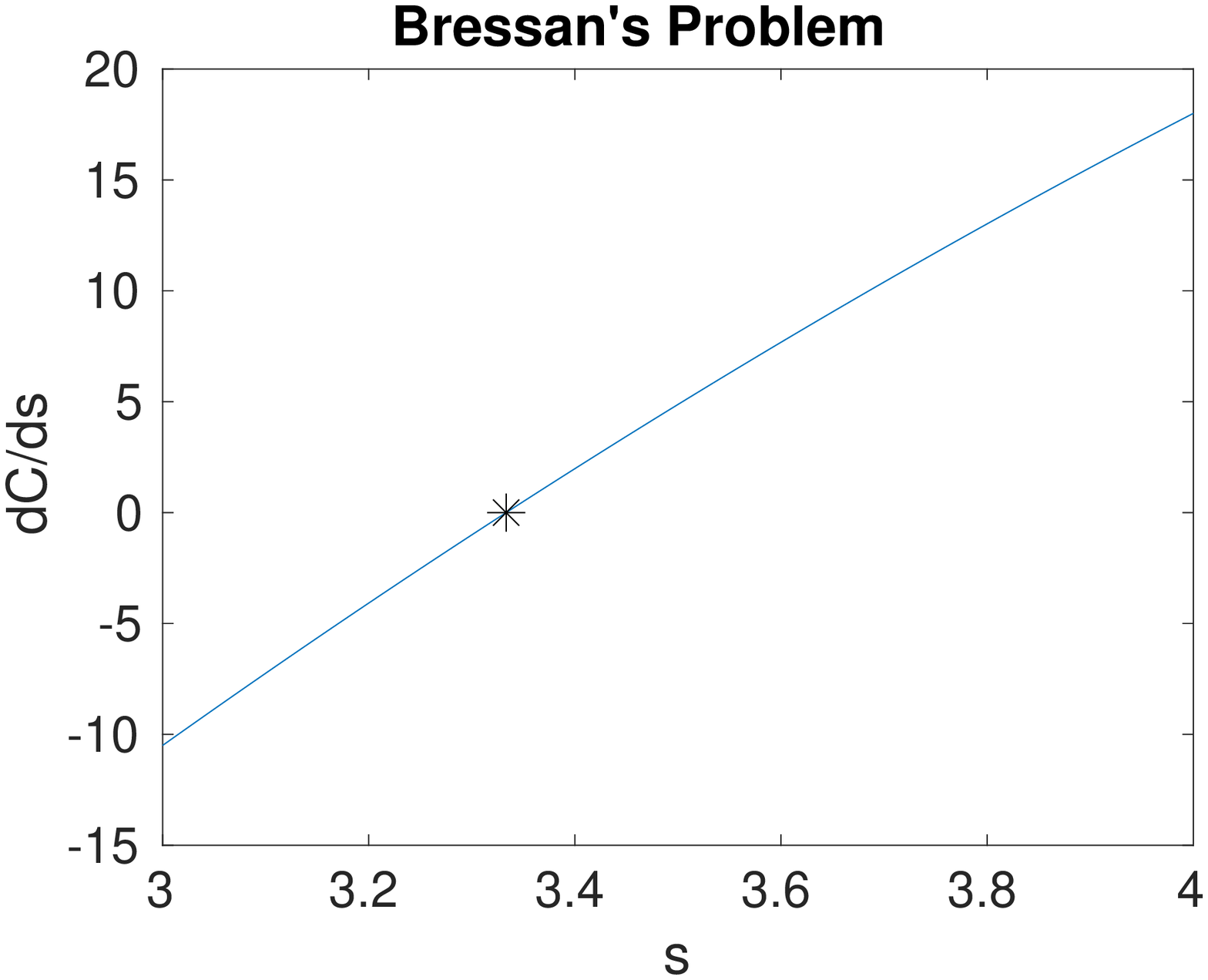}
\caption{Derivative of the objective versus switch point
$($zeros marked by stars$)$.
\label{DC}}%
\end{nscenter}
\end{figure}

\subsection{Bressan's Problem \cite{Bressan11}}
The test problem is
\[
\begin{split}
\min~ &\int_0^{T} x_1^2(t) - x_2 (t) ~ dt \\
\mbox{s.t.} ~&\dot{x}_1(t)=u(t), \quad \dot{x}_2(t)=-x_1 (t),\\
&x_1 (0) = 0,~ x_2(0) = 0, ~ x_3 (0) = 0, \quad -1 \leq u(t) \leq 1 ,
\end{split}
\]
where $T = 10$ in \cite{Biegler19b, Biegler16}.
The Hamiltonian is 
\[
H(\m{x}, u, \m{p}) = x_1^2 - x_2 + p_1 u - p_2 x_1 .
\]
The switching function is $\C{S}(t) = p_1(t)$ where $\m{p}$ is the
solution of the system
\[
\begin{array}{ll}
\dot{p}_1(t) = p_2(t)  - 2x_1(t), & p_1(T) = 0, \\
\dot{p}_2(t) = 1,                 & p_2(T) = 0. \\
\end{array}
\]
The first two derivatives of the switching function are
$\dot{\C{S}}(t) = p_2(t)  - 2x_1(t)$ and $\ddot{\C{S}}(t) = 1 - 2u(t)$.
In the singular region, $\ddot{\C{S}} = 0$, which implies that
$u(t) = 1/2$.
It can be shown \cite{Bressan11} that there is one switching point $s_1 = T/3$.
The optimal control is
\[
u(t) = 
\begin{cases} 
   -1 &  \text{if} ~ 0 \le t< s_1, \\
  1/2 &  \text{if} ~ s_1 < t \le T.
\end{cases}
\]

Unlike Jacobson's problem,
the plot of the objective derivative versus the switch point
(right panel of Figure~\ref{DC}) is nearly linear in a wide interval
around the switch point 10/3 when $T = 10$.
Starting from the initial iterates 3.0 and 4.0, the secant method converges
to 15 digit accuracy in 5 iterations.
\begin{table}[h]
\begin{center}%
\begin{tabular} [c] {|c|c|l|}\hline
Method & CPU (s)& $s_1$ Error \\ \hline
SPA1 & 0.09 & 1.8$\times 10^{-15}$\\
\cite{Biegler19b}& $0.29  $&3.3$\times 10^{-03}$ \\
\cite{Biegler16} & 24.80 & 5.9$\times 10^{-07}$\\
\hline
\end{tabular}
\end{center}
\caption{Performance and absolute errors for Bressan's problem.}
\label{data3}
\end{table}

\subsection{Goddard's Problem \cite{Goddard1919}}
A classic problem with a free terminal time is Goddard's rocket problem.
It is difficult to compare to other algorithms in the literature
since there are many variations of Goddard's problem
\cite{Betts2010, BonnansMartinon2008, Bryson75, Foroozandeh17a,
Foroozandeh17b, Martinon09, Maurer1976, Rao:2010:ACM, Vossen2010}, and different
algorithms are tested using different versions of the problem.
The formulation of the Goddard Problem that we use is based on
parameters and constraints
from both \cite[p.~213]{Betts2010} and \cite[Example 3]{Rao:2010:ACM}:
\[
\begin{array}{llr}
\min & -h(T) \\[.05in]
\mbox{s.t.} & \dot{h}(t) = v(t) & h(0) = 0 \\[.05in]
& \dot{v}(t) = \frac{1}{m} \left[ u (t) - \sigma v^2 e^{-h(t)/h_0}
- g \right] &
v(0) = 0 \\[.05in]
& \dot{m}(t) = -u(t)/c & m(0) = 3 \\[.05in]
& m(t) \ge 1, \quad 0 \le u(t) \le u_{\max}, \quad t \in [0, T],
\end{array}
\]
where $u_{\max} = 193$, $g = 32.174$, $\sigma = 5.4915 \times 10^{-5}$,
$c = 1580.9425$, and $h_0 = 23800$.
In this problem, $h$ is the height of a rocket, $v$ is its velocity, $m$
is the mass, $c$ is the exhaust velocity of the propellant,
$g$ is the gravitational acceleration, $h_0$ is the
density scale height, $u$ is the thrust, and the terminal time is free.
The goal is to choose $T$ and the thrust $u$ so as to achieve the
greatest possible terminal height for the rocket.
The mass has a lower bound (the weight of the rocket minus the weight of
the propellant) which is taken to be one.

To solve the Goddard rocket problem with the Switch Point Algorithm, we
will convert the state constraint $m \ge 1$ into an additional cost
term in the objective.
Observe that the mass is a monotone decreasing function of time
since $u \ge 0$.
Since the goal is to achieve the great possible height, all the fuel will
be consumed and $m(T) = 1$.
Consequently, the state constraint is satisfied when the terminal constraint
$m(T) = 1$ is fulfilled.
By adding a term of the form $\beta (m(t) -1)$ to the objective, where
$\beta$ corresponds to the optimal costate evaluated at the terminal time,
the solution of the Goddard problem becomes a stationary point of the
problem with the modified objective and with the state constraint omitted.
To achieve an objective where the solution to the Goddard problem
is a local minimizer, we need to incorporate a penalty term in the objective:
\begin{equation}\label{obj}
-h(T) + \beta (m(T) - 1) + \frac{\rho}{2} (m(T) -1)^2 ,
\end{equation}
with $\rho > 0$ and $\beta \approx -2.31774080357308 \times 10^{4}$.
Our Goddard test problem corresponds to the original Goddard
problem, but with the state constraint dropped and with the objective
replaced by (\ref{obj}).

The optimal solution of the Goddard rocket problem has two switch points
and the optimal control is
\begin{equation}\label{switch}
u(t) = \left\{ \begin{array}{lll}
u_{\max} & 0 \le t \le s_1 & \approx 13.75532627577406, \\
u_{\mbox{\scriptsize sing}}(t) & s_1 < t \le s_2 & \approx 21.98890645593362, \\
0 & s_2 < t \le T & \approx 42.88910958027504,
\end{array} \right.
\end{equation}
where the control in the singular region, gotten from the second derivative
of the switching function, is
\[
u_{\mbox{\scriptsize sing}}(t) =
\sigma v^2(t) e^{h(t)/h_0} + mg + \frac{mg}{1 + 4 \kappa(t) + 2\kappa^2(t)}
\left[ \frac{c^2}{h_0 g} (1 + \kappa^{-1}(t)) - 1 - 2\kappa(t) \right],
\]
where $\kappa (t) = c/v(t)$.
The switching points in (\ref{switch}) were estimated
by integrating forward the dynamics utilizing the known structure for
the optimal solution.

We solve the Goddard test problem using $\rho = 10^{5}$, the starting guess
$s_1 = 13$, $s_2 = 21$, and $T = 42$, and the optimizer PASA with
MATLAB's integrator ODE45 and computing tolerances $10^{-8}$.
The solution time was 1.21~s, and the absolute errors in $s_1$, $s_2$,
and $T$ were $1.3 \times 10^{-8}$, $6.0 \times 10^{-8}$, and
$9.4 \times 10^{-8}$ respectively.
PASA achieved the convergence tolerance in 11 iterations, and essentially
all the computing time is associated with ODE45.


\section{Conclusions}
\label{conclusions}

A new approach, the Switch Point Algorithm, was introduced for solving
nonsmooth optimal control problems whose solutions are bang-bang,
singular, or both bang-bang and singular.
For problems where the optimal control has the form
$\m{u}(t) = \g{\phi}_j(\m{x}(t),t)$ (Case~1) or
$\m{u}(t) = \g{\phi}_j(\m{x}(t), \m{p}(t),t)$ (Case~2) for
$t \in (s_j, s_{j+1})$, $0 \le j \le k$, with
$\m{u}(t)$ feasible for $\m{s}$ in a neighborhood of the
optimal switching points and for the initial costate $\m{p}(0)$
in a neighborhood of the associated optimal costate,
the solution of the optimal control problem reduces to an optimization over
the switching points and the choice of the initial costate $\m{p}(0)$.
Formulas were obtained for the derivatives of the objective with respect
to the switch points, the initial costate $\m{p}(0)$, and the terminal time $T$.
All these derivatives can be computed simultaneously
with just one integration of the state or generalized state dynamics,
followed by one integration of the costate or generalized costate dynamics.
A series of test problems were solved by either optimizing
over the switching points (and over $\m{p}(0)$ in Case~2)
or by computing a point where the derivative of the objective with respect
a switch point vanishes.
Accurate solutions were obtained relatively quickly as seen in the comparisons
given in Tables~\ref{data1}--\ref{data3}.
\bibliographystyle{siam}

\end{document}